\documentclass{amsart}

\newtheorem{theorem}{Theorem}[section]
\newtheorem{corollary}[theorem]{Corollary}
\newtheorem{proposition}[theorem]{Proposition}
\theoremstyle{definition}
\newtheorem{definition}[theorem]{Definition}
\newtheorem{example}[theorem]{Example}
\theoremstyle{remark}
\newtheorem{remark}[theorem]{Remark}
\numberwithin{equation}{section}

\makeatletter
\def\@biblabel#1{#1.}
\makeatother

\begin{document}

\title[Nonsymmetric and Symmetric Fractional Calculi]{Nonsymmetric
and Symmetric Fractional Calculi on Arbitrary Nonempty Closed Sets}

\thanks{This is a preprint of a paper whose final and definite form
will be published in \emph{Mathematical Methods in the Applied Sciences},
ISSN 0170-4214. Submitted 05/March/2014; revised 24/Feb/2015; accepted 25/Feb/2015.}

\author[N. Benkhettou]{Nadia Benkhettou}
\address{Nadia Benkhettou,
Laboratoire de Math\'{e}matiques,
Universit\'{e} de Sidi Bel-Abb\`{e}s,
B.P. 89, 22000 Sidi Bel-Abb\`{e}s, Algerie}
\email{benkhettou\_na@yahoo.fr}

\author[A. M. C. Brito da Cruz]{Artur M. C. Brito da Cruz}
\address{Artur M. C. Brito da Cruz,
Escola Superior de Tecnologia de Set\'{u}bal, Estefanilha, 2910-761
Set\'{u}bal, Portugal \and
Center for Research and Development in Mathematics and Applications (CIDMA),
University of Aveiro, 3810--193 Aveiro, Portugal}
\email{artur.cruz@estsetubal.ips.pt}

\author[D. F. M. Torres]{Delfim F. M. Torres}
\address{Delfim F. M. Torres,
Center for Research and Development in Mathematics and Applications (CIDMA),
Department of Mathematics, University of Aveiro, 3810--193 Aveiro, Portugal}
\email{delfim@ua.pt}

% -------------------------------------------------

\date{}

% -------------------------------------------------

\begin{abstract}
We introduce a nabla, a delta, and a symmetric fractional calculus
on arbitrary nonempty closed subsets of the real numbers. These
fractional calculi provide a study of differentiation and integration
of noninteger order on discrete, continuous, and hybrid settings.
Main properties of the new fractional operators are investigated,
and some fundamental results presented, illustrating the interplay
between discrete and continuous behaviors.
\end{abstract}


\maketitle

% -------------------------------------------------

\section{Introduction}

The notion of \emph{derivative} is at the core of any calculus.
One can interpret the derivative in a geometrical way, as the slope of a curve,
or, physically, as a rate of change. But what if we generalize
the notion of derivative and we study the limit
$$
\lim_{s\rightarrow t}\frac{f(s)-f(t)}{(s-t)^{\alpha}}
$$
for $\alpha \in ]0,1]$ (derivative of order $\alpha$)?
In this paper we discuss this question
in the general framework of the calculus on time scales,
which might be best understood as the continuum bridge
between discrete time and continuous time theories,
offering a rich formalism for studying hybrid
discrete-continuous dynamical systems
\cite{Bohner:1,Bohner:2,MR2994055}.
The particular case of the $q$-scale fractional
calculus is well studied in the literature: see
\cite{MR0247389,MR2963764,MR2350094} and references therein.
Here we are interested to deal with arbitrary time scales.

A time scale is a model of time, where the continuous and the discrete
are considered and merged into a single theory.
Time scales were first introduced by Aulbach and Hilger in 1988 \cite{Hilger}.
Applications in many different fields that require simultaneous
modeling of discrete and continuous data are available \cite{Ck,Dryl,Sheng}.

In order to define an inverse operator of our new derivative, the antiderivative,
we apply some ideas from fractional calculus, which is a branch of mathematical
analysis that studies the possibility of taking real number powers
of the differentiation operator \cite{book:Benchohra,Kilbas,Samko}.
The fractional calculus goes back to Leibniz (1646--1716) himself.
However, it was only in the last 20 years that fractional
calculus has gained an increasingly attention of researchers.
In October 2009, Science Watch of Thomson Reuters identified it as an
\textit{Emerging Research Front} and gave an award
to Metzler and Klafter for their paper \cite{Metzler}.
Here we consider fractional calculus in the more general setting of time scales.
The study of fractional calculus on arbitrary time scales was introduced in the PhD thesis
of Bastos \cite{PhD:Bastos,MR2728463,MyID:179,MR2800417}
and is now subject of strong current research: see, e.g.,
\cite{MyID:296,MR2809039,MR3243775,MR2798773,MR3178922,MR3270279}.
However, to the best of our knowledge, all the previous results
refer to nonsymmetric fractional calculi. Here we introduce a
general symmetric fractional calculus on time scales. For the
importance to study such a symmetric calculus we refer the reader
to \cite{MR3110279,MR3110294,MR3031158}.

The paper is organized as follows. In Section~\ref{Preliminaries}
we present the basic notions and necessary results to what follows.
Our results appear in Section~\ref{Main Results}. We begin
to define and develop the nonsymmetric fractional calculus (Section~\ref{NSFC}).
In order to do that, we define the nabla fractional derivative
and the nabla fractional integral of order $\alpha \in ]0,1]$. Then,
in Section~\ref{SFC}, we introduce and develop the symmetric fractional calculus.
We end with Section~\ref{sec:Conc} of conclusions and future work.

% -------------------------------------------------

\section{Preliminaries}
\label{Preliminaries}

A nonempty closed subset of $\mathbb{R}$ is called a time scale and is
denoted by $\mathbb{T}$. We assume that a time scale has the topology
inherited from $\mathbb{R}$ with the standard topology. Two jump operators
are considered: the forward jump operator $\sigma :\mathbb{T}\rightarrow
\mathbb{T}$, defined by $\sigma \left( t\right) :=\inf \left\{ s\in \mathbb{T}:
s>t\right\} $ with $\inf \emptyset =\sup \mathbb{T}$ (i.e., $\sigma \left(
M\right) =M$ if $\mathbb{T}$ has a maximum $M$), and the backward jump
operator $\rho :\mathbb{T}\rightarrow \mathbb{T}$ defined by $\rho \left(
t\right) :=\sup \left\{ s\in \mathbb{T}:s<t\right\} $ with $\sup \emptyset
=\inf \mathbb{T}$ (i.e., $\rho \left( m\right) =m$ if $\mathbb{T}$ has a
minimum $m$). A point $t\in \mathbb{T}$ is said to be right-dense,
right-scattered, left-dense or left-scattered if $\sigma \left( t\right) =t$,
$\sigma \left( t\right) >t$, $\rho \left( t\right) =t$ or $\rho \left(
t\right) <t$, respectively. A point $t\in \mathbb{T}$ is dense if it is
right and left dense; isolated if it is right and left scattered. If $\sup
\mathbb{T}$ is finite and left-scattered, then we define $\mathbb{T}^{\kappa
}:=\mathbb{T}\setminus \{\sup \mathbb{T}\}$, otherwise $\mathbb{T}^{\kappa
}:=\mathbb{T}$; if $\inf \mathbb{T}$ is finite and right-scattered, then
$\mathbb{T}_{\kappa }:=\mathbb{T}\setminus \{\inf \mathbb{T}\}$, otherwise
$\mathbb{T}_{\kappa }:=\mathbb{T}$. We set $\mathbb{T}_{\kappa }^{\kappa }
:=\mathbb{T}_{\kappa }\cap \mathbb{T}^{\kappa }$ and we denote $f\circ \sigma$
by $f^{\sigma }$ and $f\circ \rho $ by $f^{\rho }$.

In time scales it is common to define two derivatives: one using the forward jump
operator, the so-called delta derivative; the other using the backward jump
operator, known as the nabla derivative.

\begin{definition}[Delta derivative \cite{Bohner:1}]
We say that a function $f:\mathbb{T}\rightarrow \mathbb{R}$ is delta
differentiable at $t\in \mathbb{T}^{\kappa }$ if there exists a number
$f^{\Delta }\left( t\right)$ such that, for all $\varepsilon >0$, there
exists a neighborhood $U$ of $t$ such that
\begin{equation*}
\left\vert f^{\sigma }\left( t\right) -f\left( s\right) -f^{\Delta }\left(
t\right) \left( \sigma \left( t\right) -s\right) \right\vert \leq
\varepsilon \left\vert \sigma \left( t\right) -s\right\vert
\end{equation*}
for all $s\in U$. We call $f^{\Delta }\left( t\right) $ the delta derivative
of $f$ at $t$ and we say that $f$ is delta differentiable if $f$ is delta
differentiable for all $t\in \mathbb{T}^{\kappa }$.
\end{definition}

\begin{definition}[Nabla derivative \cite{Bohner:1}]
We say that a function $f:\mathbb{T}\rightarrow \mathbb{R}$ is nabla
differentiable at $t\in \mathbb{T}_{\kappa }$ if there exists a number
$f^{\nabla }\left( t\right)$ such that, for all $\varepsilon >0$, there
exists a neighborhood $V$ of $t$ such that
\begin{equation*}
\left\vert f^{\rho }\left( t\right) -f\left( s\right) -f^{\nabla }\left(
t\right) \left( \rho \left( t\right) -s\right) \right\vert \leq \varepsilon
\left\vert \rho \left( t\right) -s\right\vert
\end{equation*}
for all $s\in V$. We call $f^{\nabla }\left( t\right) $ the nabla derivative
of $f$ at $t$ and we say that $f$ is nabla differentiable if $f$ is nabla
differentiable for all $t\in \mathbb{T}_{\kappa }$.
\end{definition}

A third derivative, the symmetric derivative on time scales, can be seen,
under certain assumptions, as a generalization of both the nabla and delta
derivatives. The symmetric calculus on time scales was recently proposed
and investigated in \cite{Cruz:1,Cruz:2,Cruz:3}. We refer the reader to these references
for the motivation to study the symmetric calculus and for a deep understanding
of the theory.

\begin{definition}[See \cite{Cruz:2}]
\label{def:sc}
We say that a function $f:\mathbb{T}\rightarrow \mathbb{R}$
is symmetric continuous at $t\in $ $\mathbb{T}$
if, for any $\varepsilon >0$, there exists a neighborhood
$U_{t}\subset \mathbb{T}$ of $t$ such that, for all
$s\in U_{t}$ for which $2t-s\in U_{t}$,
one has $\left\vert f\left( s\right)
-f\left( 2t-s\right) \right\vert \leq \varepsilon$.
\end{definition}

Continuity implies symmetric continuity but
the reciprocal is not true \cite{Cruz:2}.

\begin{definition}[See \cite{Cruz:2}]
Let $f:\mathbb{T}\rightarrow \mathbb{R}$ and $t\in \mathbb{T}_{\kappa}^{\kappa }$.
The symmetric derivative of $f$ at $t$, denoted by
$f^{\diamondsuit }\left( t\right)$, is the real number, provided it exists,
with the property that, for any $\varepsilon >0$, there exists a
neighborhood $U\subset \mathbb{T}$ of $t$ such that
\begin{multline*}
\left\vert \left[ f^{\sigma }\left( t\right) -f\left( s\right) +f\left(
2t-s\right) -f^{\rho }\left( t\right) \right] -f^{\diamondsuit }\left(
t\right) \left[ \sigma \left( t\right) +2t-2s-\rho \left( t\right) \right]
\right\vert \\
\leq \varepsilon \left\vert \sigma \left( t\right) +2t-2s-\rho
\left( t\right) \right\vert
\end{multline*}
for all $s\in U$ for which $2t-s\in U$. A function $f$ is said to be
symmetric differentiable provided $f^{\diamondsuit }\left( t\right) $ exists
for all $t\in \mathbb{T}_{\kappa }^{\kappa }$.
\end{definition}

In time scales one uses special classes of functions
that guarantee the existence of an antiderivative.

\begin{definition}
A function $f:\mathbb{T}\rightarrow \mathbb{R}$ is called regulated provided
its right-sided limit exist (finite) at all right-dense points in $\mathbb{T}
$ and its left-sided limits exist (finite) at all left-dense points in
$\mathbb{T}$.
\end{definition}

\begin{definition}
A function $f:\mathbb{T}\rightarrow \mathbb{R}$ is called rd-continuous
provided it is continuous at right-dense points in $\mathbb{T}$ and its
left-sided limits exist (finite) at left-dense points in $\mathbb{T}$. The
set of rd-continuous functions $f:\mathbb{T}\rightarrow \mathbb{R}$ is
denoted by $\mathcal{C}_{rd}$. Analogously, a function
$f:\mathbb{T}\rightarrow \mathbb{R}$ is called ld-continuous if it is continuous at all
left-dense points and if its right-sided limits exist and are finite at all
right-dense points. We denote the set of all ld-continuous functions by
$\mathcal{C}_{ld}$.
\end{definition}

\begin{theorem}[See \cite{Bohner:1,Bohner:2}]
Every ld-continuous function $f:\mathbb{T}\rightarrow \mathbb{R}$ has a
nabla antiderivative. In particular, if $t_{0}\in \mathbb{T}$, then $F$
defined by
\begin{equation*}
F\left( t\right) :=\int_{t_{0}}^{t}f\left( \tau \right) \nabla \tau \ \
\text{for}\ \ t\in \mathbb{T}
\end{equation*}
is a nabla antiderivative x
of $f$.
\end{theorem}

For the notions of nabla and delta integrals and their generalizations,
we refer the reader to \cite{Bohner:1,MR2562284,MR2601883}.

% -------------------------------------------------

\section{Main Results}
\label{Main Results}

We develop two types of fractional calculi on arbitrary time scales: nonsymmetric
(Section~\ref{NSFC}) and symmetric (Section~\ref{SFC}). The new calculi provide,
as particular cases, discrete, quantum, continuous and hybrid fractional derivatives
and integrals.

% -------------------

\subsection{Nonsymmetric fractional calculus}
\label{NSFC}

We begin by introducing a new notion: the nabla fractional derivative of
order $\alpha \in ]0,1]$ for functions defined on arbitrary time scales. For
$\alpha =1$ we obtain the usual nabla derivative of the time-scale calculus.
Let $\mathbb{T}$ be a time scale, $t\in \mathbb{T}$, and $\delta >0$. We
define the right $\delta $-neighborhood of $t$ as $\mathcal{U}^{+}:=\left[
t,t+\delta \right[ \cap \mathbb{T}$ and the left $\delta$-neighborhood of
$t$ as $\mathcal{U}^{-}:=\left] t-\delta ,t\right] \cap \mathbb{T}$.

\begin{definition}[The nabla fractional derivative]
Let $f:\mathbb{T}\rightarrow \mathbb{R}$, $t\in \mathbb{T}_{k}$.
For $\alpha \in ]0,1]\cap \left\{1/q : q\text{ is an odd number}\right\}$
(resp. $\alpha \in ]0,1]\setminus \left\{ 1/q:q\text{ is an odd number}\right\}$)
we define $f^{\nabla^{\alpha}}(t)$ to be the number
(provided it exists) with the property that, given any $\varepsilon >0$,
there is a $\delta $-neighborhood $\mathcal{U}\subset \mathbb{T}$ of $t$
(resp. right $\delta $-neighborhood $\mathcal{U}^{+}\subset \mathbb{T}$ of
$t$), $\delta >0$, such that
\begin{equation*}
\left\vert \left[ f(s)-f^{\rho }(t))\right] -f^{\nabla ^{\alpha }}(t)\left[
s-\rho (t)\right] ^{\alpha }\right\vert \leq \varepsilon \left\vert s-\rho
(t)\right\vert ^{\alpha }
\end{equation*}
for all $s\in \mathcal{U}$ (resp. $s\in \mathcal{U}^{+}$). We call
$f^{\nabla ^{\alpha }}(t)$ the nabla fractional derivative of $f$ of order
$\alpha$ at $t$.
\end{definition}

Throughout the paper we only consider the nonsymmetric fractional
calculus as the calculus derived from the nabla fractional derivative.
However, we could also consider the delta fractional calculus
associated with the delta fractional derivative.

\begin{definition}[The delta fractional derivative]
Let $f:\mathbb{T}\rightarrow \mathbb{R}$, $t\in \mathbb{T}^{\kappa }$.
For $\alpha \in ]0,1]\cap \left\{ 1/q:q\text{ is an odd
number}\right\}$ (resp. $\alpha \in ]0,1]\setminus \left\{ 1/q:q\text{ is an
odd number}\right\} $) we define $f^{\Delta ^{\alpha }}(t)$ to be the number
(provided it exists) with the property that, given any $\varepsilon >0$,
there is a $\delta $-neighborhood $\mathcal{U}\subset \mathbb{T}$ of $t$
(resp. left $\delta $-neighborhood $\mathcal{U}^{-}\subset \mathbb{T}$ of $t$),
$\delta >0$, such that
\begin{equation*}
\left\vert \left[ f^{\sigma }(t)-f(s)\right] -f^{\Delta^{\alpha}}(t)\left[
\sigma (t)-s\right]^{\alpha }\right\vert \leq \varepsilon \left\vert \sigma
(t)-s\right\vert ^{\alpha }
\end{equation*}
for all $s\in \mathcal{U}$ (resp. $s\in \mathcal{U}^{-}$). We call
$f^{\Delta^{\alpha}}(t)$ the delta fractional derivative of $f$ of order
$\alpha$ at $t$.
\end{definition}

Along the text $\alpha$ is a real number in the interval $]0,1]$.
The next theorem provides some useful properties of the nabla
fractional derivative on time scales.

\begin{theorem}
\label{T1}Assume $f:\mathbb{T}\rightarrow \mathbb{R}$ and let
$t\in \mathbb{T}_{k}$. The following properties hold:

\begin{enumerate}
\item[(i)] Let $\alpha \in ]0,1]\cap \left\{ \frac{1}{q}:q\text{ is an odd
number}\right\} $. If $t$ is left-dense and if $f$ is nabla fractional
differentiable of order $\alpha $ at $t$, then $f$ is continuous at $t$.

\item[(ii)] Let $\alpha \in ]0,1]\setminus \left\{ \frac{1}{q}:q\text{ is an
odd number}\right\} $. If $t$ is left-dense and if $f$ is nabla fractional
differentiable of order $\alpha $ at $t$, then $f$ is right-continuous at $t$.

\item[(iii)] If $f$ is continuous at $t$ and $t$ is left-scattered, then $f$
is nabla fractional differentiable of order $\alpha $ at $t$ with
\begin{equation*}
f^{\nabla ^{\alpha }}(t)=\frac{f(t)-f^{\rho }(t)}{\left[ t-\rho (t)\right]
^{\alpha }}.
\end{equation*}

\item[(iv)] Let $\alpha \in ]0,1]\cap \left\{ \frac{1}{q}:q\text{ is an odd
number}\right\} $. If $t$ is left-dense, then $f$ is nabla fractional
differentiable of order $\alpha $ at $t$ if, and only if, the limit
\begin{equation*}
\lim_{s\rightarrow t}\frac{f(s)-f(t)}{(s-t)^{\alpha }}
\end{equation*}
exists as a finite number. In this case,
\begin{equation*}
f^{\nabla ^{\alpha }}(t)=\lim_{s\rightarrow t}\frac{f(s)-f(t)}{(s-t)^{\alpha
}}.
\end{equation*}

\item[(v)] Let $\alpha \in ]0,1]\setminus \left\{ \frac{1}{q}:q\text{ is an
odd number}\right\} $. If $t$ is left-dense, then $f$ is nabla fractional
differentiable of order $\alpha $ at $t$ if, and only if, the limit
\begin{equation*}
\lim_{s\rightarrow t^{+}}\frac{f(s)-f(t)}{(s-t)^{\alpha}}
\end{equation*}
exists as a finite number. In this case,
\begin{equation*}
f^{\nabla ^{\alpha }}(t)
=\lim_{s\rightarrow t^{+}}\frac{f(s)-f(t)}{(s-t)^{\alpha}}.
\end{equation*}

\item[(vi)] If $f$ is nabla fractional differentiable
of order $\alpha $ at $t$, then
$$
f(t)=f^{\rho }(t)+\left[ t-\rho(t)\right]^{\alpha}
f^{\nabla^{_{^{\alpha }}}}(t).
$$
\end{enumerate}
\end{theorem}

\begin{proof}
$(i)$ Assume that $f$ is fractional differentiable at $t$. Then, there
exists a neighborhood $\mathcal{U}$ of $t$ such that
\begin{equation*}
\left\vert \left[ f(s)-f^{\rho }(t)\right] -f^{\nabla ^{\alpha }}(t)\left[
s-\rho (t)\right] ^{\alpha }\right\vert \leq \varepsilon \left\vert s-\rho
(t)\right\vert ^{\alpha }
\end{equation*}
for $s\in \mathcal{U}$. Therefore, for all $s\in \mathcal{U}\cap \left]
t-\varepsilon ,t+\varepsilon \right[$,
\begin{equation*}
\begin{split}
\left\vert f\left( t\right) -f\left( s\right) \right\vert \leq & \left\vert
\left[ f(s)-f^{\rho }(t)\right] -f^{\nabla ^{\alpha }}(t)\left[ s-\rho (t)
\right]^{\alpha }\right\vert \\
&+\left\vert \left[ f(t)-f^{\rho }(t)\right]
-f^{\nabla^{\alpha }}(t)\left[ t-\rho (t)\right] ^{\alpha }\right\vert \\
&+\left\vert f^{\nabla ^{\alpha }}(t)\right\vert \left\vert \left[ s-\rho
(t)\right] ^{\alpha }-\left[ t-\rho (t)\right] ^{\alpha }\right\vert
\end{split}
\end{equation*}
and, since $t$ is a left-dense point,
\begin{equation*}
\begin{split}
\left\vert f\left( t\right) -f\left( s\right) \right\vert & \leq \left\vert
\left[ f(s)-f^{\rho }(t)\right] -f^{\nabla ^{\alpha }}(t)\left[ s-\rho (t)
\right]^{\alpha }\right\vert +\left\vert f^{\nabla ^{\alpha }}(t)\left[ s-t
\right]^{\alpha }\right\vert \\
& \leq \varepsilon \left\vert s-t\right\vert ^{\alpha }+\left\vert f^{\nabla
^{\alpha }}\left( t\right) \left[ s-t\right] ^{\alpha }\right\vert \\
& \leq \varepsilon ^{\alpha }\left[ \varepsilon +\left\vert f^{\nabla
^{\alpha }}(t)\right\vert \right] .
\end{split}
\end{equation*}
We conclude that $f$ is continuous at $t$.
$(ii)$ The proof is similar to the proof of $(i)$, where instead of
considering the neighborhood $\mathcal{U}$ of $t$ we consider a right
neighborhood $\mathcal{U}^{+}$ of $t$.
$(iii)$ Assume that $f$ is continuous at $t$ and $t$ is left-scattered. By
continuity,
\begin{equation*}
\lim_{s\rightarrow t}\frac{f(s)-f^{\rho }(t)}{\left[ s-\rho (t)\right]
^{\alpha }}=\frac{f(t)-f^{\rho }(t)}{\left[ t-\rho (t)\right] ^{\alpha }}.
\end{equation*}
Hence, given $\varepsilon >0$ and $\alpha \in ]0,1]\cap \left\{ 1/q:q\text{
is an odd number}\right\} $, there is a neighborhood $\mathcal{U}$ of $t$ (or
$\mathcal{U}^{+}$ if $\alpha \in ]0,1]\setminus \left\{ 1/q:q\text{ is an odd
number}\right\} $) such that
\begin{equation*}
\left\vert \frac{f(s)-f^{\rho }(t)}{\left[ s-\rho (t)\right] ^{\alpha }}
-\frac{f(t)-f^{\rho }(t)}{\left[ t-\rho (t)\right] ^{\alpha }}\right\vert
\leq \varepsilon
\end{equation*}
for all $s\in \mathcal{U}$ (resp. $\mathcal{U}^{+}$). It follows that
\begin{equation*}
\left\vert \left[ f(s)-f^{\rho }(t)\right] -\frac{f(t)-f^{\rho }(t)}{\left[
t-\rho (t)\right] ^{\alpha }}\left[ s-\rho (t)\right] ^{\alpha }\right\vert
\leq \varepsilon \left\vert s-\rho (t)\right\vert ^{\alpha }
\end{equation*}
for all $s\in \mathcal{U}$ (resp. $\mathcal{U}^{+}$). Hence, we get the
desired result:
\begin{equation*}
f^{\nabla ^{\alpha }}(t)=\frac{f(t)-f^{\rho }(t)}{\left[ t-\rho (t)\right]
^{\alpha }}.
\end{equation*}
$(iv)$ Assume that $f$ is nabla fractional differentiable of order $\alpha $
at $t$ and $t$ is left-dense. Let $\varepsilon >0$ be given. Since $f$ is
nabla fractional differentiable of order $\alpha $ at $t$, there is a
neighborhood $\mathcal{U}$ of $t$ such that
\begin{equation*}
\left\vert \left[ f(s)-f^{\rho }(t)\right] -f^{\nabla ^{\alpha }}(t)\left[
s-\rho (t)\right] ^{\alpha }\right\vert \leq \varepsilon \left\vert s-\rho
(t)\right\vert ^{\alpha }
\end{equation*}
for all $s\in \mathcal{U}$. Since $\rho (t)=t$,
\begin{equation*}
\left\vert \lbrack f(s)-f(t)]-f^{\nabla ^{\alpha }}(t)\left[ s-t\right]
^{\alpha }\right\vert \leq \varepsilon |s-t|^{\alpha }
\end{equation*}
for all $s\in \mathcal{U}$. It follows that
\begin{equation*}
\left\vert \frac{f(s)-f(t)}{\left[ s-t\right] ^{\alpha }}
-f^{\nabla ^{\alpha}}(t)\right\vert \leq \varepsilon
\end{equation*}
for all $s\in \mathcal{U}$, $s\neq t$. Therefore, we get the desired result:
\begin{equation*}
f^{\nabla ^{_{^{\alpha }}}}(t)=\lim_{s\rightarrow t}
\frac{f(t)-f(s)}{(t-s)^{\alpha }}.
\end{equation*}
Now assume that
\begin{equation*}
\lim_{s\rightarrow t}\frac{f(s)-f(t)}{\left( s-t\right) ^{\alpha }}
\end{equation*}
exists and is equal to $L$ and $t$ is left-dense. Then, there exists
a neighborhood $\mathcal{U}$ of $t$ such that
\begin{equation*}
\left\vert \frac{f(s)-f(t)}{\left( s-t\right) ^{\alpha }}-L\right\vert \leq
\varepsilon
\end{equation*}
for all $s\in \mathcal{U}\backslash\{t\}$. Because $t$ is left-dense,
\begin{equation*}
\left\vert \frac{f(s)-f^{\rho }(t)}{\left[ s-\rho (t)\right]^{\alpha}}
-L\right\vert \leq \varepsilon .
\end{equation*}
Therefore,
\begin{equation*}
\left\vert \left[ f(s)-f^{\rho }(t)\right]
-L\left[ s-\rho (t)\right]^{\alpha }\right\vert
\leq \varepsilon |s-\rho (t)|^{\alpha }
\end{equation*}
for all $s\in U$ (note that the inequality is trivially verified for $s=t$).
Hence, $f$ is nabla fractional differentiable
of order $\alpha $ at $t$ and
$$
f^{\nabla ^{\alpha }}(t)=\lim_{s\rightarrow t}\frac{f(s)-f(t)}{\left( s-t\right) ^{\alpha }}.
$$
$(v)$ The proof is similar to the proof of $(iv)$,
where instead of considering the neighborhood
$\mathcal{U}$ of $t$ we consider a right-neighborhood
$\mathcal{U}^{+}$ of $t$.
$(vi)$ If $\rho (t)=t$, then
\begin{equation*}
f^{\rho }(t)=f(t)=f(t)+\left[ t-\rho (t)\right]^{\alpha }f^{\nabla
^{_{^{\alpha }}}}\left( t\right).
\end{equation*}
On the other hand, if $t>\rho \left( t\right)$, then, by $(iii)$,
\begin{equation*}
f(t)=f^{\rho }(t)+\left[ t-\rho (t)\right] ^{\alpha }\frac{f(t)-f^{\rho }(t)}{
\left[ t-\rho (t)\right] ^{\alpha }}=f^{\rho }(t)
+\left[ t-\rho (t)\right]^{\alpha }f^{\nabla ^{_{^{\alpha }}}}(t).
\end{equation*}
The proof is complete.
\end{proof}

Next result relates different orders of
the nabla fractional derivative of a function.

\begin{theorem}
\label{T2}Let $\alpha ,\beta \in \left] 0,1\right] $ with $\beta \geq \alpha
$ and let $f:\mathbb{T}\rightarrow \mathbb{R}$ be a continuous function. If
$f$ is nabla fractional differentiable of order $\beta $ at $t\in \mathbb{T}$,
then $f$ is nabla fractional differentiable of order $\alpha $ at $t$.
\end{theorem}

\begin{proof}
If $t$ is left-scattered, then, by Theorem \ref{T1} $\left( iii\right) ,$ $f$
is nabla fractional differentiable of any order $\alpha \in \left] 0,1\right]$.
If $t$ is left-dense, then, by Theorem \ref{T1} $(iv,v)$,
\begin{equation*}
f^{\nabla ^{\beta }}(t)=\lim_{s\rightarrow t}\frac{f(s)-f(t)}{(s-t)^{\beta }}.
\end{equation*}
Since
\begin{equation*}
f^{\nabla ^{\beta }}(t)=\lim_{s\rightarrow t}
\frac{\frac{f(s)-f(t)}{(s-t)^{\alpha }}}{(s-t)^{\beta -\alpha }},
\end{equation*}
we have
\begin{equation*}
f^{\nabla ^{\alpha }}(t)=\lim_{s\rightarrow t}(s-t)^{\beta -\alpha}
f^{\nabla ^{\beta }}(t),
\end{equation*}
which proves existence of the nabla fractional derivative of $f$ of
order $\alpha$ at $t\in \mathbb{T}$.
\end{proof}

\begin{proposition}
\label{E1:i} If $f:\mathbb{T}\rightarrow \mathbb{R}$ is defined by $f(t)=c$
for all $t\in \mathbb{T}$, $c\in \mathbb{R}$, then $f^{\nabla ^{\alpha
}}\equiv 0$.
\end{proposition}

\begin{proof}
If $t$ is left-scattered, then, by Theorem~\ref{T1} $(iii)$, one has
\begin{equation*}
f^{\nabla ^{\alpha }}(t)=\frac{f(t)-f^{\rho }(t)}{\left[ t-\rho (t)\right]
^{\alpha }}=\frac{c-c}{\left[ t-\rho (t)\right] ^{\alpha }}=0.
\end{equation*}
Assume $t$ is left-dense. Then, by Theorem~\ref{T1} $(iv)$ and $(v)$, it follows
that
\begin{equation*}
f^{\nabla ^{\alpha }}(t)=\lim_{s\rightarrow t}\frac{c-c}{\left[
t-\rho(t)\right] ^{\alpha }}=0.
\end{equation*}
This concludes the proof.
\end{proof}

\begin{proposition}
\label{E1:ii} If $f:\mathbb{T}\rightarrow \mathbb{R}$ is defined by $f(t)=t$
for all $t\in \mathbb{T}$, then
\begin{equation*}
f^{\nabla ^{\alpha }}(t)
=
\begin{cases}
\left[ t-\rho (t)\right] ^{1-\alpha } & \text{ if }\alpha \neq 1, \\
1 & \text{ if }\alpha =1.
\end{cases}
\end{equation*}
\end{proposition}

\begin{proof}
Clearly, the function is nabla differentiable, which is the same as saying
that function $f$ is nabla fractional differentiable of order 1. Then, by
Theorem~\ref{T2}, the function is nabla fractional differentiable of order
$\alpha$, with $\alpha \in \left] 0,1\right] $. From Theorem~\ref{T1} $(vi)$
it follows that
\begin{equation*}
t-\rho (t)=\left[ t-\rho (t)\right]^{\alpha }f^{\nabla ^{_{^{\alpha }}}}(t).
\end{equation*}
If $t-\rho (t)\neq 0$, then
$$
f^{\nabla ^{_{^{\alpha }}}}(t)=\left[ t-\rho (t)\right] ^{1-\alpha }$$
and the desired relation is proved. Assume now that $t-\rho (t)=0$, that is,
$\rho (t)=t$. In this case $t$ is left-dense and by Theorem~\ref{T1} $(iv)$ and
$(v)$ it follows that
\begin{equation*}
f^{\nabla ^{\alpha }}(t)=\lim_{s\rightarrow t}\frac{s-t}{(s-t)^{\alpha }}.
\end{equation*}
Therefore, if $\alpha =1$, then $f^{\nabla ^{\alpha }}(t)=1$; if $0<\alpha
<1 $, then $f^{\nabla ^{\alpha }}(t)=0$.
\end{proof}

Let us now consider the particular case $\mathbb{T}=\mathbb{R}$.

\begin{corollary}
Function $f:\mathbb{R}\rightarrow \mathbb{R}$ is nabla fractional
differentiable of order $\alpha $ at point $t\in \mathbb{R}$ if, and only
if, the limit
\begin{equation*}
\lim_{s\rightarrow t}\frac{f(s)-f(t)}{(s-t)^{\alpha }}
\end{equation*}
exists as a finite number. In this case,
\begin{equation*}
f^{\nabla ^{\alpha }}(t)=\lim_{s\rightarrow t}\frac{f(s)-f(t)}{(s-t)^{\alpha
}}.
\end{equation*}
\end{corollary}

\begin{proof}
Here $\mathbb{T}=\mathbb{R}$ and all points are left-dense. The result
follows from Theorem~\ref{T1} $(iv)$ and $(v)$. Note that if $\alpha \in
]0,1]\setminus \left\{ \frac{1}{q}:q\text{ is an odd number}\right\} $, then
the limit only makes sense as a right-side limit.
\end{proof}

The next result shows that there are functions which are nabla fractional
differentiable but are not nabla differentiable.

\begin{proposition}
\label{prop:ex}
If $f:\mathbb{R}_{0}^{+}\rightarrow \mathbb{R}$
is defined by $f(t)=\sqrt{t}$ for all $t\in \mathbb{R}_{0}^{+}$, then
\begin{equation*}
f^{\nabla^{1/2}}(t)
=
\begin{cases}
0 & \text{ if }t\neq 0, \\
1 & \text{ if }t=0.
\end{cases}
\end{equation*}
\end{proposition}

\begin{proof}
Here $\mathbb{T} = \mathbb{R}_{0}^{+}$.
In this time scale every point $t$
is left-dense and by Theorem~\ref{T1} $(v)$
with $\alpha = 1/2$ it follows that
\begin{equation*}
f^{\nabla^{1/2}}(t)
=\lim_{s\rightarrow t^{+}}\frac{\sqrt{s}
-\sqrt{t}}{\sqrt{s-t}}
=\lim_{s\rightarrow t^{+}}\frac{\sqrt{s-t}}{\sqrt{s}+\sqrt{t}}
=0
\end{equation*}
for $t\neq 0$. If $t=0$, then
\begin{equation*}
f^{\nabla ^{1/2}}(t)
=\lim_{s\rightarrow 0^{+}}\frac{\sqrt{s}}{\sqrt{s}}=1.
\end{equation*}
This concludes the proof.
\end{proof}

\begin{remark}
If $f: \mathbb{R} \rightarrow \mathbb{R}$ is differentiable, then
\[
f^{\nabla ^{\alpha }}\left( t\right) =\lim_{s\rightarrow t}\frac{f\left(
s\right)-f\left( t\right)}{\left(s-t\right)^{\alpha }}
=\lim_{s\rightarrow t}\frac{f\left( s\right) -f\left( t\right) }{s-t}\left(
s-t\right) ^{1-\alpha }=f^{\prime }\left( t\right) \lim_{s\rightarrow
t}\left( s-t\right) ^{1-\alpha }=0.
\]
Therefore, the nabla derivative of order $\alpha $ is
always zero whenever the function is differentiable.
The fractional nabla derivative is particular useful to study the behaviour
of functions that are not differentiable in the classical sense, as shown in Proposition~\ref{prop:ex}.
\end{remark}

For the fractional derivative on time scales to be useful, we would like to
know formulas for the derivatives of sums, products and quotients of
fractional differentiable functions. This is done according to the following
theorem.

\begin{theorem}
\label{T3} Assume $f,g:\mathbb{T}\rightarrow \mathbb{R}$ are nabla
fractional differentiable of order $\alpha $ at $t\in \mathbb{T}_{k}$. Then,

\begin{enumerate}
\item[(i)] the sum $f+g:\mathbb{T}\rightarrow \mathbb{R}$ is nabla
fractional differentiable at $t$ with $$(f+g)^{\nabla ^{\alpha
}}(t)=f^{\nabla ^{\alpha }}(t)+g^{\nabla ^{\alpha }}(t);$$

\item[(ii)] for any constant $\lambda \in
\mathbb{R}
$, $\lambda f:\mathbb{T}\rightarrow \mathbb{R}$ is nabla fractional
differentiable at $t$ with $$(\lambda f)^{\nabla ^{\alpha }}(t)=\lambda
f^{\nabla ^{\alpha }}(t);$$

\item[(iii)] if $f$ and $g$ are continuous, then the product $fg:\mathbb{T}
\rightarrow \mathbb{R}$ is nabla fractional differentiable at $t$ with
\begin{equation*}
\begin{split}
(fg)^{\nabla ^{\alpha }}(t)& =f^{\nabla ^{\alpha }}\left( t\right) g\left(
t\right) +f^{\rho }\left( t\right) g^{\nabla ^{\alpha }}\left( t\right) \\
& =f^{\nabla ^{\alpha }}(t)g^{\rho }(t)+f(t)g^{\nabla ^{\alpha }}(t);
\end{split}
\end{equation*}

\item[(iv)] if $f$ is continuous and $f^{\rho }(t)f(t)\neq 0$,
then $\frac{1}{f}$ is nabla fractional differentiable at $t$ with
\begin{equation*}
\left( \frac{1}{f}\right) ^{\nabla ^{\alpha }}(t)=-\frac{f^{\nabla ^{\alpha
}}(t)}{f^{\rho }(t)f(t)};
\end{equation*}

\item[(v)] if $f$ and $g$ are continuous and $g^{\rho }(t)g(t)\neq 0$, then
$\frac{f}{g}$ is fractional differentiable at $t$ with
\begin{equation*}
\left( \frac{f}{g}\right) ^{\nabla ^{\alpha }}(t)=\frac{f^{\nabla ^{\alpha
}}(t)g(t)-f(t)g^{\nabla ^{\alpha }}(t)}{g^{\rho }(t)g(t)}.
\end{equation*}
\end{enumerate}
\end{theorem}

\begin{proof}
Let us consider that $\alpha \in ]0,1]\cap \left\{ \frac{1}{q}:q\text{ is an
odd number}\right\} $. The proofs for the case $\alpha \in ]0,1]\setminus
\left\{ \frac{1}{q}:q\text{ is an odd number}\right\} $ are similar: one just
needs to choose the proper right-sided neighborhoods. Assume that $f$ and $g$
are nabla fractional differentiable of order $\alpha$ at $t\in \mathbb{T}_{k}$.
$(i)$ Let $\varepsilon >0$. Then there exist neighborhoods
$\mathcal{U}_{1}$ and $\mathcal{U}_{2}$ of $t$ for which
\begin{equation*}
\left\vert \left[ f(s)-f^{\rho }(t)\right] -f^{\nabla ^{\alpha }}(t)\left[
s-\rho (t)\right] ^{\alpha }\right\vert \leq \frac{\varepsilon }{2}
\left\vert s-\rho (t)\right\vert ^{\alpha }~~for~all~~s\in \mathcal{U}_{1}
\end{equation*}
and
\begin{equation*}
\left\vert \left[ g(s)-g^{\rho }(t)\right] -g^{\nabla ^{\alpha }}(t)\left[
s-\rho (t)\right] ^{\alpha }\right\vert \leq \frac{\varepsilon }{2}
\left\vert s-\rho (t)\right\vert ^{\alpha }~~for~all~~s\in \mathcal{U}_{2}.
\end{equation*}
Let $\mathcal{U}=\mathcal{U}_{1}\cap \mathcal{U}_{2}$. Then
\begin{equation*}
\begin{split}
\biggl|(&f+g)(s) -(f +g)^{\rho }(t)-\left[
f^{\nabla ^{\alpha }}(t)+g^{\nabla ^{\alpha }}(t)\right] \left[ s-\rho (t)
\right] ^{\alpha }\biggr| \\
& \leq \left\vert \left[ f(s)-f^{\rho }(t)\right] -f^{\nabla ^{\alpha }}(t)
\left[ s-\rho (t)\right] ^{\alpha }\right\vert +\left\vert \left[
g(s)-g^{\rho }(t)\right] -g^{\nabla ^{\alpha }}(t)\left[ s-\rho (t)\right]
^{\alpha }\right\vert \\
& \leq \varepsilon |s-\rho \left( t\right)|^{\alpha}
\end{split}
\end{equation*}
for all $s\in \mathcal{U}$. Therefore, $f+g$ is fractional differentiable
of order $\alpha$ at
$t$ and
$$
(f+g)^{\nabla ^{\alpha }}(t)=f^{\nabla ^{\alpha }}(t)+g^{\nabla
^{\alpha }}(t).
$$
$(ii)$ Let $\varepsilon >0$. Then there exists a
neighborhood $\mathcal{U}$ of $t$ with
\begin{equation*}
\left\vert \left[ f(s)-f^{\rho }(t)\right] -f^{\nabla ^{\alpha }}(t)\left[
s-\rho (t)\right] ^{\alpha }\right\vert \leq \varepsilon \left\vert s-\rho
(t)\right\vert ^{\alpha }\text{ for all }s\in \mathcal{U}.
\end{equation*}
It follows that
\begin{equation*}
\left\vert \left[ \left( \lambda f\right) (s)-\left( \lambda f\right) ^{\rho
}(t)\right] -\lambda f^{\nabla ^{\alpha }}(t)\left[ s-\rho (t)\right]
^{\alpha }\right\vert \leq \varepsilon |\lambda |\,\left\vert s-\rho
(t)\right\vert |^{\alpha }\text{ for all }s\in \mathcal{U}.
\end{equation*}
Therefore, $\lambda f$ is fractional differentiable of order $\alpha$  at $t$ and $(\lambda
f)^{\nabla ^{\alpha }}(t)=\lambda f^{\nabla ^{\alpha }}(t)$ holds at $t$.
$(iii)$ If $t$ is left-dense, then
\begin{equation*}
\begin{split}
(fg)^{\nabla ^{\alpha }}(t)& =\lim_{s\rightarrow t}\frac{\left( fg\right)
(s)-\left( fg\right) (t)}{(s-t)^{\alpha }} \\
& =\lim_{s\rightarrow t}\frac{f(s)-f(t)}{\left( s-t\right)^{\alpha}}
g\left( s\right) +\lim_{s\rightarrow t}\frac{g(s)-g(t)}{(s-t)^{\alpha}}
f\left( t\right) \\
& =f^{\nabla^{\alpha}}\left( t\right) g\left( t\right) +f^{\rho }\left(
t\right) g^{\nabla ^{\alpha }}\left( t\right) .
\end{split}
\end{equation*}
If $t$ is left-scattered, then
\begin{equation*}
\begin{split}
\left( fg\right) ^{\nabla ^{\alpha }}(t)& =\frac{\left( fg\right) (t)-\left(
fg\right) ^{\rho }(t)}{[t-\rho (t)]^{\alpha }} \\
& =\frac{f(t)-f^{\rho }(t)}{[t-\rho \left( t\right) ]^{\alpha }}g\left(
t\right) +\frac{g(t)-g^{\rho }(t)}{[t-\rho (t)]^{\alpha }}f^{\rho }(t) \\
& =f^{\nabla ^{\alpha }}(t)g(t)+f^{\rho }(t)g^{\nabla ^{\alpha }}(t).
\end{split}
\end{equation*}
The other product rule formula follows
by interchanging the role of functions $f$ and $g$.
$(iv)$ Using the fractional derivative of a constant
(Proposition~\ref{E1:i}) and Theorem~\ref{T3} $(iii)$, we know that
\begin{equation*}
\left( f\cdot \frac{1}{f}\right) ^{\nabla ^{\alpha }}(t)=(1)^{(\alpha )}(t)=0.
\end{equation*}
Therefore,
\begin{equation*}
\left( \frac{1}{f}\right)^{\nabla ^{\alpha }}(t)f^{\rho }(t)
+f^{\nabla^{\alpha}}(t)\frac{1}{f(t)}=0.
\end{equation*}
Since we are assuming $f^{\rho }(t)\neq 0$,
\begin{equation*}
\left( \frac{1}{f}\right)^{\nabla ^{\alpha }}(t)
=-\frac{f^{\nabla ^{\alpha}}(t)}{f^{\rho }(t)f(t)},
\end{equation*}
as intended.
$(v)$ Follows trivially from the previous properties:
\begin{equation*}
\begin{split}
\left( \frac{f}{g}\right) ^{\nabla ^{\alpha }}(t)& =\left( f\cdot \frac{1}{g}
\right)^{\nabla ^{\alpha }}(t) \\
& =f(t)\left( \frac{1}{g}\right) ^{\nabla ^{\alpha }}(t)+f^{\nabla ^{\alpha
}}(t)\frac{1}{g^{\rho }(t)} \\
& =-f(t)\frac{g^{\nabla ^{\alpha }}(t)}{g^{\rho }(t)g\left( t\right)}
+f^{\nabla ^{\alpha }}(t)\frac{1}{g^{\rho }(t)} \\
& =\frac{f^{\nabla ^{\alpha }}(t)g(t)
-f(t)g^{\nabla^{\alpha }}(t)}{g(t)g(\sigma (t))}.
\end{split}
\end{equation*}
The proof is complete.
\end{proof}

The next result provides examples of how to use the algebraic properties of
the nabla fractional derivatives of order $\alpha$.

\begin{theorem}
\label{thm:der:pf}
Let $c$ be a constant, $m\in \mathbb{N}$,
and $\alpha \in \left] 0,1\right[$.
\begin{enumerate}
\item[(i)] If $\displaystyle f(t)=(t-c)^{m}$, then
\begin{equation*}
f^{\nabla ^{\alpha }}(t)
=
\begin{cases}
\displaystyle\left[ t-\rho (t)\right] ^{1-\alpha }\sum_{\nu =0}^{m-1}\left[
\rho (t)-c\right] ^{\nu }(t-c)^{m-1-\nu } & \text{ if }\alpha \neq 1, \\
\displaystyle\sum_{\nu =0}^{m-1}\left[ \rho(t)
-c\right]^{\nu}(t-c)^{m-1-\nu } & \text{ if }\alpha =1.
\end{cases}
\end{equation*}

\item[(ii)] If $\displaystyle g(t)=\frac{1}{(t-c)^{m}}$, then
\begin{equation*}
g^{\nabla^{\alpha }}(t)
=
\begin{cases}
\displaystyle-\left[ t-\rho (t)\right] ^{1-\alpha }\sum_{\nu =0}^{m-1}
\frac{1}{\left[ \rho (t)-c\right] ^{m-\nu }(t-c)^{\nu +1}}
& \text{ if }\alpha \neq 1, \\
\displaystyle-\sum_{\nu =0}^{m-1}\frac{1}{\left[ \rho(t)
-c\right]^{m-\nu}(t-c)^{\nu +1}} & \text{ if }\alpha =1,
\end{cases}
\end{equation*}
provided $\left[ \rho (t)-c\right] (t-c)\neq 0$.
\end{enumerate}
\end{theorem}

\begin{proof}
We use mathematical induction. First, let us consider
the case $\alpha \neq 0$. If $m=1$, then $f(t)=t-c$ and
$f^{\nabla^{\alpha }}(t)=\left[ t-\rho (t)\right]^{1-\alpha }$ holds from
Propositions~\ref{E1:i} and \ref{E1:ii} and Theorem~\ref{T3} $(i)$. Now,
assume that
\begin{equation*}
f^{\nabla ^{\alpha }}(t)=\left[ t-\rho (t)\right]^{1-\alpha }
\sum_{\nu=0}^{m-1}\left[ \rho (t)-c\right] ^{\nu }(t-c)^{m-1-\nu}
\end{equation*}
holds for $f(t)=(t-c)^{m}$ and let $F(t)=(t-c)^{m+1}=(t-c)f(t)$.
By Theorem~\ref{T3} $(iii)$, we have
\begin{equation*}
\begin{split}
F^{\nabla ^{\alpha }}(t)&=(t-c)^{\nabla ^{\alpha }}f^{\rho }(t)
+f^{\nabla^{\alpha }}(t)(t-c)\\
&=\left[ t-\rho (t)\right] ^{1-\alpha }f^{\rho}(t)
+f^{\nabla ^{\alpha }}(t)(t-c)\\
&=\left[ t-\rho (t)\right] ^{1-\alpha }\left[ \left[ \rho (t)-c\right]
^{m}+\sum_{\nu =0}^{m-1}\left[ \rho (t)-c\right] ^{\nu }(t-c)^{m-\nu }\right]\\
&=\left[ t-\rho (t)\right] ^{1-\alpha }\sum_{\nu =0}^{m}\left[ \rho (t)
-c\right] ^{\nu }(t-c)^{m-\nu}.
\end{split}
\end{equation*}
Hence, by mathematical induction, part $(i)$ holds. For $\displaystyle g(t)
=\frac{1}{(t-c)^{m}}=\frac{1}{f(t)}$, we apply Theorem~\ref{T3} $(iv)$ to obtain
\begin{equation*}
\begin{split}
g^{\nabla ^{\alpha }}(t)& =-\frac{f^{\nabla ^{\alpha }}(t)}{f^{\rho }(t)f(t)}
=-\left[ t-\rho (t)\right] ^{1-\alpha }\frac{\sum_{\nu =0}^{m-1}\left[ \rho
(t)-c\right] ^{\nu }(t-c)^{m-1-\nu }}{\left[ \rho (t)-c\right] ^{m}(t-c)^{m}}
\\
& =-\left[ t-\rho (t)\right] ^{1-\alpha }\sum_{\nu =0}^{m-1}\frac{1}{\left[
\rho (t)-c\right] ^{m-\nu }(t-c)^{\nu +1}},
\end{split}
\end{equation*}
provided $\left[ \rho (t)-c\right] (t-c)\neq 0$. For $\alpha =1$ the proofs
are similar bearing in mind that
$\left( t\right)^{\nabla^{1}}=\left(t\right)^{\nabla }=1$.
\end{proof}

Now we introduce the nabla fractional integral on time scales.

\begin{definition}[The indefinite nabla fractional integral]
\label{def:int}
Assume that $f:\mathbb{T}\rightarrow \mathbb{R}$ is a
regulated function. We define the indefinite nabla fractional integral of $f$
of order $\beta$, $0\leq \beta \leq 1$, by
\begin{equation*}
\int f(t)\nabla^{\beta }t
:=\left(\int f(t)\nabla t\right)^{\nabla^{(1-\beta )}}
\end{equation*}
with $\int f(t)\nabla t$ the usual indefinite nabla integral of time
scales \cite{Bohner:1}.
\end{definition}

\begin{remark}
It follows from Definition~\ref{def:int} that
$$
\int f(t)\nabla ^{1}t=\int f(t)\nabla t,
\quad \int f(t)\nabla^{0}t=f(t).
$$
\end{remark}

\begin{definition}[The definite nabla fractional integral]
\label{def:intFracCauchy}
Assume $f:\mathbb{T}\rightarrow \mathbb{R}$ is a ld-continuous function. Let
\begin{equation*}
F^{\nabla ^{\beta }}(t)=\int f(t)\nabla ^{\beta }t
\end{equation*}
denote the indefinite nabla fractional integral of $f$ of order $\beta$
with $0\leq \beta \leq 1$, and let $a,b\in \mathbb{T}$.
We define the Cauchy nabla fractional integral from $a$ to $b$ by
\begin{equation*}
\int_{a}^{b}f(t)\nabla ^{\beta }t
:=\left. F^{\nabla ^{\beta }}(t)\right\vert_{a}^{b}
=F^{\nabla ^{\beta }}(b)-F^{\nabla ^{\beta }}(a).
\end{equation*}
\end{definition}

The next theorem gives some algebraic properties of the nabla fractional integral.

\begin{theorem}
\label{T4}
If $a,b,c\in \mathbb{T}$, $\lambda \in \mathbb{R}$, and $f,g\in
\mathcal{C}_{ld}$ with $0\leq \beta \leq 1$, then

\begin{enumerate}
\item[(i)] $\displaystyle\int_{a}^{b}[f(t)+g(t)]\nabla ^{\beta}t
=\int_{a}^{b}f(t)\nabla ^{\beta }t+\int_{a}^{b}g(t)\nabla ^{\beta }t$;

\item[(ii)] $\displaystyle\int_{a}^{b}(\lambda f)(t)\nabla ^{\beta}t
=\lambda \int_{a}^{b}f(t)\nabla ^{\beta }t$;

\item[(iii)] $\displaystyle\int_{a}^{b}f(t)\nabla ^{\beta}t
=-\int_{b}^{a}f(t)\nabla ^{\beta }t$;

\item[(iv)] $\displaystyle\int_{a}^{b}f(t)\nabla ^{\beta}t
=\int_{a}^{c}f(t)\nabla ^{\beta }t+\int_{c}^{b}f(t)\nabla ^{\beta }t$;

\item[(v)] $\displaystyle\int_{a}^{a}f(t)\nabla ^{\beta }t=0$.
\end{enumerate}
\end{theorem}

\begin{proof}
The equalities follow from Definitions~\ref{def:int} and \ref{def:intFracCauchy},
analogous properties of the nabla integral on time scales,
and the properties of Section~\ref{NSFC} for the fractional
nabla derivative on time scales. $(i)$
From Definition~\ref{def:intFracCauchy}, we have
\begin{equation*}
\begin{split}
\int_{a}^{b}(f+g)(t)\nabla ^{\beta }t & =\left. \int \left[ f(t)+g(t)\right]
\nabla ^{\beta }t\right\vert _{a}^{b} \\
& =\left. \left( \int \left[ f(t)+g(t)
\right] \nabla t\right) ^{\nabla ^{(1-\beta )}}\right\vert _{a}^{b}\\
& =\int_{a}^{b}f(t)\nabla ^{\beta }t+\int_{a}^{b}g(t)\nabla ^{\beta }t.
\end{split}
\end{equation*}
$(ii)$ From Definitions~\ref{def:intFracCauchy} and \ref{def:int}, one has
\begin{equation*}
\begin{split}
\int_{a}^{b}(\lambda f)(t)\nabla ^{\beta }t
&=\left. \int (\lambda f)(t)\nabla^{\beta }t\right\vert _{a}^{b}
=\left.\left( \int (\lambda
f)(t)\nabla t\right)^{\nabla^{(1-\beta )}}\right\vert_{a}^{b}\\
&=\left. \lambda \left(\int f(t)
\nabla t\right)^{\nabla^{(1-\beta)}}\right\vert_{a}^{b}
=\lambda \int_{a}^{b}f(t)\nabla^{\beta }t.
\end{split}
\end{equation*}
The last properties $(iii)$, $(iv)$ and $(v)$
are direct consequences of
Definition~\ref{def:intFracCauchy}: $(iii)$
\begin{equation*}
\int_{a}^{b}f(t)\nabla ^{\beta }t=F^{\beta }(b)-F^{\beta }(a)
=-\left(F^{\beta }(a)-F^{\beta }(b)\right)
=-\int_{b}^{a}f(t)\nabla ^{\beta }t;
\end{equation*}
$(iv)$
\begin{equation*}
\begin{split}
\int_{a}^{b}f(t)\nabla ^{\beta }t& =F^{\nabla ^{\beta }}(b)-F^{\nabla
^{\beta }}(a)=F^{\nabla ^{\beta }}(c)-F^{\nabla ^{\beta }}(a)+F^{\nabla
^{\beta }}(b)-F^{\nabla ^{\beta }}(c) \\
& =\int_{a}^{c}f(t)\nabla ^{\beta }t+\int_{c}^{b}f(t)\nabla ^{\beta }t;
\end{split}
\end{equation*}
$(v)$
\begin{equation*}
\int_{a}^{a}f(t)\nabla ^{\beta }t
=F^{\nabla ^{\beta }}(a)-F^{\nabla^{\beta}}(a)=0.
\end{equation*}
This concludes the proof.
\end{proof}

We end this section with a simple example
of a discrete fractional integral of order $\alpha$.

\begin{example}
Let $\mathbb{T}=\mathbb{Z}$,
$0\leq \beta < 1$, and $f(t)=t$.
Using the fact that
\begin{equation*}
\int t\nabla t=\frac{t^{2}}{2}+C,
\end{equation*}
where $C$ is a constant, we have
\begin{equation*}
\int_{1}^{10}t\,\nabla ^{\beta }t=\left. \int t\,\nabla ^{\beta
}t\right\vert _{1}^{10}=\left. \left( \int t\,\nabla t\right) ^{\nabla
^{(1-\beta )}}\right\vert _{1}^{10}=\left. \left( \frac{t^{2}}{2}+C\right)
^{(1-\beta )}\right\vert _{1}^{10}.
\end{equation*}
It follows that
\begin{equation*}
\int_{1}^{10}t\,\nabla ^{\beta }t=\left.[
t-\rho (t)]^{(1-\alpha)} [\rho(t)+t]\right\vert _{1}^{10}
=\left. \frac{1}{2}\left( 2t-1\right)
\right\vert _{1}^{10}=\frac{19}{2}-\frac{1}{2}=9.
\end{equation*}
\end{example}

The fundamental concepts of the nabla fractional calculus of order $\alpha$,
which are the differentiation and integration of noninteger order
using the nabla operator, were presented in this section. Instead of using the nabla
operator, we could use the delta operator and we would obtain an analogous delta
fractional calculus of order $\alpha$. The properties of the delta
fractional calculus of order $\alpha$ are similar to the properties of the nabla case.
In the next section we will use both nabla and delta approaches,
the delta and nabla fractional calculi of order $\alpha$,
to obtain useful results of the symmetric fractional calculus of order $\alpha$
and extend the results of \cite{Cruz:2}.

% -------------------

\subsection{Symmetric fractional calculus}
\label{SFC}

Let us now introduce the notion of symmetric fractional
derivative of order $\alpha\in]0,1]$ on time scales.

\begin{definition}[The symmetric fractional derivative]
Let $f:\mathbb{T}\rightarrow \mathbb{R}$, $t\in \mathbb{T}_{\kappa }^{\kappa}$,
and $\alpha \in \left] 0,1\right]$. The symmetric fractional derivative
of $f$ at $t$, denoted by $f^{\diamondsuit ^{\alpha }}\left( t\right) $, is
the real number (provided it exists) with the property that, for any
$\varepsilon >0$, there exists a neighborhood $U\subset \mathbb{T}$ of $t$
such that
\begin{multline*}
\left\vert \left[ f^{\sigma }\left( t\right) -f\left( s\right) +f\left(
2t-s\right) -f^{\rho }\left( t\right) \right]-f^{\diamondsuit^{\alpha}}\left( t\right)
\left[ \sigma \left( t\right) +2t-2s-\rho \left( t\right)\right]^{\alpha }\right\vert\\
\leq \varepsilon \left\vert \sigma \left(t\right)
+2t-2s-\rho \left( t\right) \right\vert^{\alpha}
\end{multline*}
for all $s\in U$ for which $2t-s\in U$. A function $f$ is said to be
symmetric fractional differentiable of order $\alpha$ provided $f^{\diamondsuit ^{\alpha
}}\left( t\right) $ exists for all $t\in \mathbb{T}_{\kappa }^{\kappa }$.
\end{definition}

\begin{remark}
If $\alpha =1$, then the symmetric fractional derivative is the symmetric
derivative on time scales \cite{Cruz:2}.
\end{remark}

Some useful properties of the symmetric derivative
are given in Theorem~\ref{ts:propriedade}.

\begin{theorem}
\label{ts:propriedade}
Let $f:\mathbb{T}\rightarrow \mathbb{R}$, $t\in\mathbb{T}_{\kappa }^{\kappa }$
and $\alpha \in \left] 0,1\right] $. The following holds:
\begin{enumerate}
\item[(i)] Function $f$ has at most one
symmetric fractional derivative of order $\alpha$.

\item[(ii)] If $f$ is symmetric fractional differentiable
of order $\alpha$ at $t$ and $t$ is dense or isolated,
then $f$ is symmetric continuous at $t$ (Definition~\ref{def:sc}).

\item[(iii)] If $f$ is continuous at $t$ and $t$ is not dense, then $f$ is
symmetric differentiable of order $\alpha$ at $t$ with
\begin{equation*}
f^{\diamondsuit ^{\alpha }}\left( t\right) =\frac{f^{\sigma }\left( t\right)
-f^{\rho }\left( t\right) }{\left[ \sigma \left( t\right) -\rho \left(
t\right) \right] ^{\alpha }}.
\end{equation*}

\item[(iv)] If $t$ is dense, then $f$
is symmetric fractional differentiable of order $\alpha$
at $t$ if and only if the limit
\begin{equation*}
\lim_{s\rightarrow t}\frac{f\left( 2t-s\right) -f\left( s\right) }{2^{\alpha
}\left( t-s\right)^{\alpha }}
\end{equation*}
exists as a finite number. In this case,
\begin{equation*}
f^{\diamondsuit ^{\alpha }}\left( t\right)
=\lim_{s\rightarrow t}\frac{f\left( 2t-s\right)
-f\left( s\right) }{2^{\alpha }\left( t-s\right)^{\alpha}}
=\lim_{h\rightarrow 0}\frac{f\left( t+h\right) -f\left(
t-h\right)}{2^{\alpha }h^{\alpha}}.
\end{equation*}

\item[(v)] If $f$ is symmetric differentiable of order $\alpha$ and continuous at $t$, then
\begin{equation*}
f^{\sigma }\left( t\right) =f^{\rho }\left( t\right) +f^{\diamondsuit
^{\alpha }}\left( t\right) \left[ \sigma \left( t\right) -\rho \left(
t\right) \right] ^{\alpha }.
\end{equation*}
\end{enumerate}
\end{theorem}

\begin{proof}
$(i)$ Suppose that $f$ has two symmetric derivatives of order $\alpha$ at $t$,
$f_{1}^{\diamondsuit ^{\alpha }}\left( t\right) $ and
$f_{2}^{\diamondsuit^{\alpha }}\left( t\right) $. Then,
there exists a neighborhood $U_{1}$ of $t$ such that
\begin{multline*}
\left\vert \left[ f^{\sigma }\left( t\right) -f\left( s\right) +f\left(
2t-s\right) -f^{\rho }\left( t\right) \right]
-f_{1}^{\diamondsuit ^{\alpha}}\left( t\right) \left[ \sigma \left( t\right)
+2t-2s-\rho \left( t\right)\right]^{\alpha }\right\vert\\
\leq \frac{\varepsilon }{2}\left\vert \sigma\left( t\right)
+2t-2s-\rho \left( t\right) \right\vert ^{\alpha }
\end{multline*}
for all $s\in U_{1}$ for which $2t-s\in U_{1}$, and a neighborhood $U_{2}$
of $t$ such that
\begin{multline*}
\left\vert \left[ f^{\sigma }\left( t\right) -f\left( s\right) +f\left(
2t-s\right) -f^{\rho }\left( t\right) \right]
-f_{2}^{\diamondsuit ^{\alpha}}\left( t\right) \left[ \sigma \left( t\right)
+2t-2s-\rho \left( t\right)\right] ^{\alpha }\right\vert\\
\leq \frac{\varepsilon }{2}\left\vert \sigma\left( t\right)
+2t-2s-\rho \left( t\right) \right\vert^{\alpha}
\end{multline*}
for all $s\in U_{2}$ for which $2t-s\in U_{2}$. Therefore, for all $s\in
U_{1}\cap U_{2}$ for which $2t-s\in U_{1}\cap U_{2}$,
\begin{equation*}
\begin{split}
\bigg{|}& f_{1}^{\diamondsuit ^{\alpha }}\left( t\right)
-f_{2}^{\diamondsuit ^{\alpha }}\left( t\right) \bigg{|}=\left\vert \left[
f_{1}^{\diamondsuit ^{\alpha }}\left( t\right) -f_{2}^{\diamondsuit ^{\alpha
}}\left( t\right) \right] \frac{\left[ \sigma \left( t\right) +2t-2s-\rho
\left( t\right) \right] ^{\alpha }}{\left[ \sigma \left( t\right)
+2t-2s-\rho \left( t\right) \right] ^{\alpha }}\right\vert \\
& \leq \frac{\left\vert \left[ f^{\sigma }\left(
t\right) -f\left( s\right) +f\left( 2t-s\right) -f^{\rho }\left( t\right)
\right] -f_{2}^{\diamondsuit ^{\alpha }}\left( t\right) \left[ \sigma \left(
t\right) +2t-2s-\rho \left( t\right) \right] ^{\alpha }\right\vert}{\left\vert
\sigma \left( t\right) +2t-2s-\rho \left(t\right) \right\vert ^{\alpha }} \\
& \qquad +\frac{\left\vert \left[ f^{\sigma }\left( t\right) -f\left( s\right)
+f\left( 2t-s\right) -f^{\rho }\left( t\right) \right] -f_{1}^{\diamondsuit
^{\alpha }}\left( t\right) \left[ \sigma \left( t\right) +2t-2s-\rho \left(
t\right) \right] ^{\alpha }\right\vert}{\left\vert
\sigma \left( t\right) +2t-2s-\rho \left(t\right) \right\vert^{\alpha }}\\
& \leq \varepsilon .
\end{split}
\end{equation*}
$(ii)$ From hypothesis, for any $\varepsilon >0$, there exists a
neighborhood $U$ of $t$ such that
\begin{multline*}
\left\vert \left[ f^{\sigma }\left( t\right) -f\left( s\right) +f\left(
2t-s\right) -f^{\rho }\left( t\right) \right]
-f^{\diamondsuit ^{\alpha}}\left( t\right) \left[ \sigma \left( t\right)
+2t-2s-\rho \left( t\right)\right]^{\alpha }\right\vert\\
\leq \varepsilon \left\vert \sigma \left(t\right)
+2t-2s-\rho \left( t\right) \right\vert^{\alpha}
\end{multline*}
for all $s\in U$ for which $2t-s\in U$. Note that
\begin{equation*}
\begin{split}
\vert f&\left( s\right) -f\left( 2t-s\right) \vert \\
&\leq \left\vert \left[ f^{\sigma }\left( t\right) -f\left( s\right)
+f\left( 2t-s\right) -f^{\rho }\left( t\right) \right] -f^{\diamondsuit
^{\alpha }}\left( t\right) \left[ \sigma \left( t\right) +2t-2s-\rho \left(
t\right) \right] ^{\alpha }\right\vert \\
&\quad +\left\vert \left[ f^{\sigma }\left( t\right) -f\left( t\right) +f\left(
t\right) -f^{\rho }\left( t\right) \right] -f^{\diamondsuit ^{\alpha
}}\left( t\right) \left[ \sigma \left( t\right) +2t-2t-\rho \left( t\right)
\right] ^{\alpha }\right\vert \\
&\quad +\left\vert f^{\diamondsuit ^{\alpha }}\left( t\right) \left[ \sigma
\left( t\right) +2t-2s-\rho \left( t\right) \right] ^{\alpha}
-f^{\diamondsuit ^{\alpha }}\left( t\right) \left[ \sigma \left( t\right)
-\rho \left( t\right) \right] ^{\alpha }\right\vert \\
&\leq \varepsilon \left\vert \sigma \left( t\right) +2t-2s-\rho \left(
t\right) \right\vert ^{\alpha }+\varepsilon \left\vert \sigma \left(
t\right) +2t-2t-\rho \left( t\right) \right\vert ^{\alpha } \\
&\quad +\left\vert f^{\diamondsuit ^{\alpha }}\left( t\right) \left[ \sigma
\left( t\right) +2t-2s-\rho \left( t\right) \right] ^{\alpha}
-f^{\diamondsuit ^{\alpha }}\left( t\right) \left[ \sigma \left( t\right)
-\rho \left( t\right) \right] ^{\alpha }\right\vert .
\end{split}
\end{equation*}
If $t$ is dense, then
\begin{equation*}
\begin{split}
\left\vert f\left( s\right) -f\left( 2t-s\right) \right\vert
&\leq \varepsilon 2^{\alpha }\left\vert t-s\right\vert ^{\alpha }+\left\vert
f^{\diamondsuit ^{\alpha }}\left( t\right) \right\vert 2^{\alpha }\left\vert
t-s\right\vert ^{\alpha } \\
&\leq \varepsilon ^{\alpha }2^{\alpha }\left( \varepsilon +\left\vert
f^{\diamondsuit ^{\alpha }}\left( t\right) \right\vert \right)
\end{split}
\end{equation*}
for all $s\in U\cap \left] t-\varepsilon, t+\varepsilon \right[$,
which proves the result for a point $t$ which is dense. If $t$ is isolated,
then the function is continuous at $t$ (because of the inherited topology)
and therefore the function is symmetric continuous at $t$.
$(iii)$ Suppose that $t\in \mathbb{T}_{\kappa }^{\kappa }$ is not
dense and $f$ is continuous at $t$. Then,
\begin{equation*}
\lim_{s\rightarrow t}\frac{f^{\sigma }\left( t\right) -f\left( s\right)
+f\left( 2t-s\right) -f^{\rho }\left( t\right) }{\left[ \sigma \left(
t\right) +2t-2s-\rho \left( t\right) \right] ^{\alpha }}=\frac{f^{\sigma
}\left( t\right) -f^{\rho }\left( t\right) }{\left[ \sigma \left( t\right)
-\rho \left( t\right) \right] ^{\alpha }}.
\end{equation*}
Hence, for any $\varepsilon >0$, there exists a neighborhood $U$ of $t$ such
that
\begin{equation*}
\left\vert \frac{f^{\sigma }\left( t\right) -f\left( s\right) +f\left(
2t-s\right) -f^{\rho }\left( t\right) }{\left[ \sigma \left( t\right)
+2t-2s-\rho \left( t\right) \right] ^{\alpha }}-\frac{f^{\sigma }\left(
t\right) -f^{\rho }\left( t\right) }{\left[ \sigma \left( t\right) -\rho
\left( t\right) \right] ^{\alpha }}\right\vert \leq \varepsilon
\end{equation*}
for all $s\in U$ for which $2t-s\in U$. It follows that
\begin{multline*}
\left\vert \left[ f^{\sigma }\left( t\right) -f\left( s\right) +f\left(
2t-s\right) -f^{\rho }\left( t\right) \right] -\frac{f^{\sigma }\left(
t\right) -f^{\rho }\left( t\right) }{\left[ \sigma \left( t\right) -\rho
\left( t\right) \right] ^{\alpha }}\left[ \sigma \left( t\right) +2t-2s-\rho
\left( t\right) \right] ^{\alpha }\right\vert \\
\leq \varepsilon \left\vert \sigma \left( t\right) +2t-2s-\rho \left(
t\right) \right\vert ^{\alpha },
\end{multline*}
which proves that
\begin{equation*}
f^{\diamondsuit ^{\alpha }}\left( t\right) =\frac{f^{\sigma }\left( t\right)
-f^{\rho }\left( t\right) }{\left[ \sigma \left( t\right) -\rho \left(
t\right) \right] ^{\alpha }}.
\end{equation*}
$(iv)$ Assume that $f$ is symmetric fractional differentiable of order $\alpha$ at $t$
and $t$ is dense. Let $\varepsilon >0$ be given. Then, there exists a
neighborhood $U$ of $t$ such that
\begin{multline*}
\left\vert \left[ f^{\sigma }\left( t\right) -f\left( s\right) +f\left(
2t-s\right) -f^{\rho }\left( t\right) \right]
-f^{\diamondsuit ^{\alpha}}\left( t\right) \left[ \sigma \left( t\right)
+2t-2s-\rho \left( t\right)\right] ^{\alpha }\right\vert\\
\leq \varepsilon \left\vert \sigma \left(t\right)
+2t-2s-\rho \left( t\right) \right\vert^{\alpha}
\end{multline*}
for all $s\in U$ for which $2t-s\in U$. Since $t$ is dense,
$$\left\vert \left[ -f\left( s\right) +f\left( 2t-s\right) \right] -f^{\diamondsuit
^{\alpha }}\left( t\right) \left[ 2t-2s\right] ^{\alpha }\right\vert \leq
\varepsilon \left\vert 2t-2s\right\vert ^{\alpha }$$
 for all $s\in U$ for
which $2t-s\in U$. It follows that
\begin{equation*}
\left\vert \frac{f\left( 2t-s\right) -f\left( s\right) }{\left( 2t-2s\right)
^{\alpha }}-f^{\diamondsuit ^{\alpha }}\left( t\right) \right\vert \leq
\varepsilon
\end{equation*}
for all $s\in U$ with $s\neq t$. Therefore, we get the desired result:
\begin{equation*}
f^{\diamondsuit^{\alpha }}\left( t\right)
=\lim_{s\rightarrow t}\frac{f\left( 2t-s\right)
-f\left( s\right) }{2^{\alpha }\left( t-s\right)^{\alpha }}.
\end{equation*}
Conversely, let us suppose that $t$ is dense and the limit
\begin{equation*}
\lim_{s\rightarrow t}\frac{f\left( 2t-s\right) -f\left( s\right) }{2^{\alpha
}\left( t-s\right) ^{\alpha }}=:L
\end{equation*}
exists. Then, there exists a neighborhood $U$ of $t$ such that $\left\vert
\frac{f\left( 2t-s\right) -f\left( s\right) }{2^{\alpha }\left( t-s\right)
^{\alpha }}-L\right\vert \leq \varepsilon $ for all $s\in U$ for which
$2t-s\in U$. Because $t$ is dense, we have
\begin{equation*}
\left\vert \frac{f^{\sigma }\left( t\right) -f\left( s\right) +f\left(
2t-s\right) -f^{\rho }\left( t\right) }{\left[ \sigma \left( t\right)
+2t-2s-\rho \left( t\right) \right] ^{\alpha }}-L\right\vert \leq
\varepsilon .
\end{equation*}
Therefore,
\begin{multline*}
\left\vert \left[ f^{\sigma }\left( t\right) -f\left( s\right) +f\left(
2t-s\right) -f^{\rho }\left( t\right) \right] -L\left[ \sigma \left(
t\right) +2t-2s-\rho \left( t\right) \right]^{\alpha }\right\vert\\
\leq \varepsilon \left\vert \sigma \left( t\right)
+2t-2s-\rho \left( t\right)\right\vert^{\alpha },
\end{multline*}
which leads us to the conclusion that $f$ is symmetric differentiable
of order $\alpha$ and $f^{\diamondsuit ^{\alpha }}\left( t\right) =L$.
Note that if we use the substitution $s=t+h$, then
\begin{equation*}
f^{\diamondsuit^{\alpha }}\left( t\right) =\lim_{h\rightarrow 0}
\frac{f\left( t+h\right) -f\left( t-h\right) }{2^{\alpha }h^{\alpha}}.
\end{equation*}
$(v)$ If $t$ is a dense point, then $\sigma \left( t\right)
=\rho\left( t\right)$ and
$$
f^{\sigma }\left( t\right) =f^{\rho }\left( t\right)
+f^{\diamondsuit ^{\alpha }}\left( t\right) \left[ \sigma \left( t\right)
-\rho \left( t\right) \right]^{\alpha}.
$$
If $t$ is not dense, and since $f$
is continuous, then
\begin{equation*}
f^{\diamondsuit ^{\alpha }}\left( t\right) =\frac{f^{\sigma }\left( t\right)
-f^{\rho }\left( t\right) }{\left[ \sigma \left( t\right) -\rho \left(
t\right) \right] ^{\alpha }}\Leftrightarrow f^{\sigma }\left( t\right)
=f^{\rho }\left( t\right) +f^{\diamondsuit }\left( t\right) \left[ \sigma
\left( t\right) -\rho \left( t\right) \right] ^{\alpha }.
\end{equation*}
This concludes the proof.
\end{proof}

\begin{remark}
\label{rem:alt:def}
An alternative way to define the symmetric fractional derivative of $f$
of order $\alpha \in \left] 0,1\right]$ at $t\in \mathbb{T}_{\kappa}^{\kappa}$
consists in saying that the limit
\begin{equation*}
\begin{split}
f^{\diamondsuit ^{\alpha }}\left( t\right)
&=\lim_{s\rightarrow t}
\frac{f^{\sigma }\left( t\right) -f\left( s\right) +f\left( 2t-s\right) -f^{\rho}
\left( t\right) }{\left[ \sigma \left( t\right) +2t-2s-\rho \left( t\right)
\right]^{\alpha }}\\
&=\lim_{h\rightarrow 0}\frac{f^{\sigma }\left( t\right)
-f\left( t+h\right) +f\left( t-h\right) -f^{\rho }\left( t\right) }{\left[
\sigma \left( t\right) -2h-\rho \left( t\right) \right] ^{\alpha }}
\end{split}
\end{equation*}
exists. Similarly, we can say that the nabla fractional derivative of $f$
of order $\alpha $ is defined by
\begin{equation*}
f^{\nabla^{\alpha }}\left( t\right) =\lim_{s\rightarrow t}
\frac{f(s)-f^{\rho }(t)}{\left[ s-\rho (t)\right]^{\alpha}}
\end{equation*}
and the delta fractional derivative of $f$ of order $\alpha $ is defined by
\begin{equation*}
f^{\Delta ^{\alpha }}\left( t\right) =\lim_{s\rightarrow t}\frac{f^{\sigma
}\left( t\right) -f\left( s\right) }{\left[ \sigma \left( t\right) -s\right]
^{\alpha }}.
\end{equation*}
\end{remark}

\begin{remark}
A function $f:\mathbb{R}\rightarrow \mathbb{R}$ is symmetric
fractional differentiable of order $\alpha $ at point $t\in \mathbb{R}$ if,
and only if, the limit
\begin{equation*}
\lim_{h\rightarrow 0}\frac{f\left( t+h\right)
-f\left( t-h\right)}{2^{\alpha }h^{\alpha }}
\end{equation*}
exists as finite number. In this case,
\begin{equation*}
f^{\diamondsuit ^{\alpha }}\left( t\right)
=\lim_{h\rightarrow 0}\frac{f\left( t+h\right)
-f\left( t-h\right) }{2^{\alpha }h^{\alpha }}.
\end{equation*}
If $\alpha=1$, then we obtain the classical symmetric derivative
in the real numbers \cite{Thomson}, defined by
\begin{equation*}
f^{\diamondsuit }\left( t\right)
=\lim_{h\rightarrow 0}\frac{f\left( t+h\right) -f\left( t-h\right) }{h}.
\end{equation*}
\end{remark}

\begin{remark}
Let $h>0$. If a function $f:h\mathbb{Z}\rightarrow \mathbb{R}$
is symmetric differentiable of order $\alpha$ for $t\in h\mathbb{Z}$, then
\begin{equation*}
f^{\diamondsuit ^{\alpha }}\left( t\right) =\frac{f\left( t+h\right)
-f\left( t-h\right) }{2^{\alpha }h^{\alpha }}.
\end{equation*}
If $\alpha=1$, then we obtain the symmetric $h$-derivative
in the quantum set $h\mathbb{Z}$ \cite{Kac}, defined by
\begin{equation*}
f^{\diamondsuit }\left( t\right) =\frac{f\left( t+h\right)
-f\left( t-h\right) }{h}.
\end{equation*}
\end{remark}

Let us now see some examples.
\begin{proposition}
If $f:\mathbb{T}\rightarrow \mathbb{R}$ is defined by $f\left( t\right) =c$
for all $t\in \mathbb{T}$, $c\in \mathbb{R}$,
then $f^{\diamondsuit ^{\alpha }}\left( t\right)=0$
for any $t\in \mathbb{T}_{\kappa}^{\kappa }$.
\end{proposition}

\begin{proof}
Trivially, we have
\begin{equation*}
f^{\diamondsuit ^{\alpha }}\left( t\right) =\lim_{s\rightarrow t}
\frac{f^{\sigma }\left( t\right) -f\left( s\right) +f\left( 2t-s\right)
-f^{\rho}\left( t\right) }{\left[ \sigma \left( t\right) +2t-2s-\rho \left( t\right)
\right]^{\alpha }}=0.
\end{equation*}
The proof is complete.
\end{proof}

\begin{proposition}
If $f:\mathbb{T}\rightarrow \mathbb{R}$ is defined by $f\left( t\right) =t$
for all $t\in \mathbb{T}$, then
\begin{equation*}
f^{\diamondsuit ^{\alpha }}\left( t\right)
=
\begin{cases}
\left[ \sigma \left( t\right) -\rho \left( t\right) \right] ^{1-\alpha }
& \text{ if } \alpha \neq 1\\
1 & \text{ if } \alpha =1
\end{cases}
\end{equation*}
for all $t\in \mathbb{T}_{\kappa }^{\kappa }$.
\end{proposition}

\begin{proof}
If $t$ is not dense, then by Theorem~\ref{ts:propriedade} (iii)
\begin{equation*}
f^{\diamondsuit ^{\alpha }}\left( t\right)
= \frac{f^{\sigma }\left(t\right)
-f^{\rho }\left( t\right) }{\left[ \sigma \left( t\right)
-\rho\left( t\right) \right] ^{\alpha }}
=\left[ \sigma \left( t\right) -\rho \left( t\right) \right] ^{1-\alpha }.
\end{equation*}
If $t$ is dense, then by Theorem~\ref{ts:propriedade} (iv)
\begin{equation*}
f^{\diamondsuit ^{\alpha }}\left( t\right)
=\lim_{s\rightarrow t}
\frac{f\left( 2t-s\right)-f\left( s\right)}{2^{\alpha }\left( t-s\right)^{\alpha }}
=\lim_{s\rightarrow t}\frac{2\left( t-s\right) }{2^{\alpha }\left(t-s\right)^{\alpha}}.
\end{equation*}
Thus, if $\alpha =1$, then $f^{\diamondsuit ^{\alpha }}\left( t\right) =1$;
if $0<\alpha <1$, then $f^{\diamondsuit ^{\alpha }}\left( t\right) =0$.
\end{proof}

Next proposition shows a function that is not differentiable
at point $t=0$ in the sense of classical (integer-order) calculus
and standard (nonsymmetric) calculus on time scales \cite{Bohner:1,Bohner:2},
but is symmetric fractional differentiable of order $\alpha \in \left] 0,1\right]$.

\begin{proposition}
\label{ex:mod}
Let $\mathbb{T}$ be a time scale with
$0\in \mathbb{T}_{\kappa }^{\kappa }$ and
$f:\mathbb{T}\rightarrow \mathbb{R}$ be defined by
$f\left( t\right) =\left\vert t\right\vert$. Then,
\begin{equation*}
f^{\diamondsuit^{\alpha }}(0)
=
\begin{cases}
0 & \text{ if } 0\text{ is dense}\\
\displaystyle\frac{\sigma (0)+\rho (0)}{\left[ \sigma (0)
-\rho (0)\right]^{\alpha }} & \text{ otherwise}
\end{cases}
\end{equation*}
for any $\alpha \in \left] 0,1\right]$.
\end{proposition}

\begin{proof}
We know (see Remark~\ref{rem:alt:def}) that
\begin{equation*}
f^{\diamondsuit }\left( 0\right) =\lim_{h\rightarrow 0}\frac{f^{\sigma
}\left( 0\right) -f\left( 0+h\right) +f\left( 0-h\right) -f^{\rho }\left(
0\right) }{\left[ \sigma \left( 0\right) -2h-\rho \left( 0\right) \right]
^{\alpha }}=\lim_{h\rightarrow 0}\frac{\sigma \left( 0\right) +\rho \left(
0\right) }{\left[ \sigma \left( 0\right) -2h-\rho \left( 0\right) \right]
^{\alpha }}.
\end{equation*}
The result follows immediately from this equality.
\end{proof}

We now give some algebric properties of the symmetric fractional derivative.

\begin{theorem}
\label{propriedades}
Let $f,g:\mathbb{T}\rightarrow \mathbb{R}$ be two symmetric fractional
differentiable functions of order $\alpha$ at
$t\in \mathbb{T}_{\kappa}^{\kappa }$
and let $\lambda \in \mathbb{R}$. The following holds:
\begin{enumerate}
\item[(i)] Function $f+g$ is symmetric fractional
differentiable of order $\alpha$ at $t$ with
\begin{equation*}
\left( f+g\right) ^{\diamondsuit ^{\alpha }}\left( t\right)
=f^{\diamondsuit^{\alpha }}\left( t\right)
+g^{\diamondsuit ^{\alpha }}\left( t\right).
\end{equation*}

\item[(ii)] Function $\lambda f$ is symmetric fractional differentiable
of order $\alpha$ at $t$ with
\begin{equation*}
\left( \lambda f\right) ^{\diamondsuit ^{\alpha }}\left( t\right)
=\lambda f^{\diamondsuit ^{\alpha }}\left( t\right).
\end{equation*}

\item[(iii)] If $f$ and $g$ are continuous at $t$, then $fg$ is symmetric
fractional differentiable of order $\alpha$ at $t$ with
\begin{equation*}
\left( fg\right) ^{\diamondsuit ^{\alpha }}\left( t\right) =f^{\diamondsuit
^{\alpha }}\left( t\right) g^{\sigma }\left( t\right) +f^{\rho }\left(
t\right) g^{\diamondsuit ^{\alpha }}\left( t\right) .
\end{equation*}

\item[(iv)] If $f$ is continuous at $t$ and $f^{\sigma }\left( t\right)
f^{\rho }\left( t\right) \neq 0$, then $1/f$ is symmetric fractional
differentiable of order $\alpha$ at $t$ with
\begin{equation*}
\left( \frac{1}{f}\right)^{\diamondsuit ^{\alpha }}\left( t\right)
=-\frac{f^{\diamondsuit^{\alpha }}\left( t\right) }{f^{\sigma }\left( t\right)
f^{\rho }\left( t\right)}.
\end{equation*}

\item[(v)] If $f$ and $g$ are continuous at $t$ and $g^{\sigma }\left(
t\right) g^{\rho }\left( t\right) \neq 0$, then $f/g$ is symmetric
fractional differentiable of order $\alpha$ at $t$ with
\begin{equation*}
\left( \frac{f}{g}\right) ^{\diamondsuit ^{\alpha }}\left( t\right)
=\frac{f^{\diamondsuit ^{\alpha }}\left( t\right) g^{\rho }\left( t\right)
-f^{\rho}\left( t\right) g^{\diamondsuit^{\alpha }}\left(
t\right)}{g^{\sigma}\left( t\right) g^{\rho }\left( t\right)}.
\end{equation*}
\end{enumerate}
\end{theorem}

\begin{proof}
$(i)$ For $t\in \mathbb{T}_{\kappa }^{\kappa }$ we have
\begin{equation*}
\begin{split}
\left( f+g\right)^{\diamondsuit ^{\alpha }}\left( t\right)
&=\lim_{s\rightarrow t}\frac{\left( f+g\right) ^{\sigma }\left( t\right)
-\left( f+g\right) \left( s\right) +\left( f+g\right) \left( 2t-s\right)
-\left( f+g\right) ^{\rho }\left( t\right) }{\left[ \sigma \left( t\right)
+2t-2s-\rho \left( t\right) \right] ^{\alpha }} \\
&=\lim_{s\rightarrow t}\frac{f^{\sigma }\left( t\right) -f\left( s\right)
+f\left( 2t-s\right) -f^{\rho }\left( t\right) }{\left[ \sigma \left(
t\right) +2t-2s-\rho \left( t\right) \right] ^{\alpha }}\\
&\quad +\lim_{s\rightarrow t}\frac{g^{\sigma }\left( t\right)
-g\left( s\right) +g\left( 2t-s\right)
-g^{\rho }\left( t\right) }{\left[ \sigma \left( t\right)
+2t-2s-\rho \left(t\right) \right] ^{\alpha }} \\
& =f^{\diamondsuit ^{\alpha }}\left( t\right)
+g^{\diamondsuit ^{\alpha}}\left( t\right).
\end{split}
\end{equation*}
$(ii)$ Let $t\in \mathbb{T}_{\kappa }^{\kappa }$ and $\lambda \in
\mathbb{R}$. Then,
\begin{equation*}
\begin{split}
\left( \lambda f\right) ^{\diamondsuit ^{\alpha }}\left( t\right) &
=\lim_{s\rightarrow t}\frac{\left( \lambda f\right) ^{\sigma }\left(
t\right) -\left( \lambda f\right) \left( s\right) +\left( \lambda f\right)
\left( 2t-s\right) -\left( \lambda f\right) ^{\rho }\left( t\right) }{\left[
\sigma \left( t\right) +2t-2s-\rho \left( t\right) \right] ^{\alpha }} \\
& =\lambda \lim_{s\rightarrow t}\frac{f^{\sigma }\left( t\right) -f\left(
s\right) +f\left( 2t-s\right) -f^{\rho }\left( t\right) }{\left[ \sigma
\left( t\right) +2t-2s-\rho \left( t\right) \right] ^{\alpha }}\\
&=\lambda f^{\diamondsuit ^{\alpha }}\left( t\right) .
\end{split}
\end{equation*}
$(iii)$ Let us assume that $t\in \mathbb{T}_{\kappa }^{\kappa }$ and
$f$ and $g$ are continuous at $t$. If $t$ is dense, then
\begin{equation*}
\begin{split}
\left( fg\right) ^{\diamondsuit ^{\alpha }}\left( t\right) &
=\lim_{h\rightarrow 0}\frac{\left( fg\right) \left( t+h\right) -\left(
fg\right) \left( t-h\right) }{2^{\alpha }h^{\alpha }} \\
& =\lim_{h\rightarrow 0}\frac{f\left( t+h\right)
-f\left(t-h\right)}{2^{\alpha }h^{\alpha }}g\left( t+h\right)
+\lim_{h\rightarrow 0}\frac{g\left( t+h\right)
-g\left( t-h\right) }{2^{\alpha }h^{\alpha }}f\left(t-h\right) \\
& =f^{\diamondsuit ^{\alpha }}\left( t\right) g^{\sigma }\left( t\right)
+f^{\rho }\left( t\right) g^{\diamondsuit ^{\alpha }}\left( t\right) .
\end{split}
\end{equation*}
If $t$ is not dense, then
\begin{equation*}
\begin{split}
\left( fg\right) ^{\diamondsuit ^{\alpha }}\left( t\right) & =\frac{\left(
fg\right) ^{\sigma }\left( t\right) -\left( fg\right) ^{\rho }\left(
t\right) }{\left[ \sigma \left( t\right) -\rho \left( t\right) \right]^{\alpha}}
=\frac{f^{\sigma }\left( t\right)
-f^{\rho }\left( t\right) }{\left[ \sigma \left( t\right)
-\rho \left( t\right) \right]^{\alpha}}g^{\sigma }\left( t\right)
+\frac{g^{\sigma }\left( t\right) -g^{\rho}\left( t\right) }{
\left[ \sigma \left( t\right)
-\rho \left( t\right)\right]^{\alpha }}f^{\rho }\left( t\right) \\
& =f^{\diamondsuit ^{\alpha }}\left( t\right) g^{\sigma }\left( t\right)
+f^{\rho }\left( t\right) g^{\diamondsuit ^{\alpha }}\left( t\right).
\end{split}
\end{equation*}
We just proved the intended equality.
$(iv)$ Using the relation
$\left( \frac{1}{f}\times f\right) \left( t\right) =1$
we can write that
\begin{equation*}
0=\left( \frac{1}{f}\times f\right) ^{\diamondsuit ^{\alpha }}\left(
t\right)=f^{\diamondsuit ^{\alpha }}\left( t\right) \left( \frac{1}{f}
\right)^{\sigma }\left( t\right) +f^{\rho }\left( t\right)
\left( \frac{1}{f}\right) ^{\diamondsuit ^{\alpha }}\left(t\right).
\end{equation*}
Therefore,
\begin{equation*}
\left(\frac{1}{f}\right)^{\diamondsuit^{\alpha }}\left( t\right)
=-\frac{f^{\diamondsuit ^{\alpha }}\left(t\right)}{f^{\sigma }\left(t\right)
f^{\rho }\left( t\right)}.
\end{equation*}
$(v)$ Let $t\in \mathbb{T}_{\kappa }^{\kappa }$. Then,
\begin{equation*}
\begin{split}
\left( \frac{f}{g}\right) ^{\diamondsuit ^{\alpha }}\left( t\right) &
=\left( f\times \frac{1}{g}\right) ^{\diamondsuit ^{\alpha }}\left( t\right)
=f^{\diamondsuit ^{\alpha }}\left( t\right) \left( \frac{1}{g}\right)^{\sigma }\left( t\right)
+f^{\rho }\left( t\right) \left( \frac{1}{g}\right)^{\diamondsuit ^{\alpha }}\left( t\right)\\
& =\frac{f^{\diamondsuit ^{\alpha }}\left( t\right) }{g^{\sigma }\left(
t\right) }+f^{\rho }\left( t\right) \left( -\frac{g^{\diamondsuit ^{\alpha
}}\left( t\right) }{g^{\sigma }\left( t\right) g^{\rho }\left( t\right)}
\right) =\frac{f^{\diamondsuit ^{\alpha }}\left( t\right) g^{\rho }\left(
t\right) -f^{\rho }\left( t\right) g^{\diamondsuit ^{\alpha }}\left(
t\right) }{g^{\sigma }\left( t\right) g^{\rho }\left( t\right) }.
\end{split}
\end{equation*}
The proof is complete.
\end{proof}

\begin{example}
The symmetric fractional derivative of $f\left( t\right) =t^{2}$ of order $\alpha$ is
\begin{equation*}
f^{\diamondsuit ^{\alpha }}\left( t\right) =\left\{
\begin{array}{ccc}
\left[ \sigma \left( t\right) -\rho \left( t\right) \right] ^{1-\alpha }
\left[ \sigma \left( t\right) +\rho \left( t\right) \right] & \text{ if } &
\alpha \neq 1 \\
\sigma \left( t\right) +\rho \left( t\right) & \text{ if } & \alpha =1.
\end{array}
\right.
\end{equation*}
\end{example}

\begin{example}
The symmetric derivative of $f\left( t\right) =1/t$ of order $\alpha$ is
\begin{equation*}
f^{\diamondsuit ^{\alpha }}\left( t\right) =\left\{
\begin{array}{ccc}
\displaystyle-\frac{\left[ \sigma \left( t\right) -\rho \left( t\right)
\right] ^{1-\alpha }}{\sigma \left( t\right) \rho \left( t\right) } & \text{
if } & \alpha \neq 1 \\[0.3cm]
\displaystyle-\frac{1}{\sigma \left( t\right) \rho \left( t\right) } & \text{
if } & \alpha =1.
\end{array}
\right.
\end{equation*}
\end{example}

The next result gives a relation between the nonsymmetric
and symmetric fractional derivatives.

\begin{proposition}
\label{delta:nabla}
If $f$ is both delta and nabla fractional differentiable
of order $\alpha$, then $f$ is symmetric fractional
differentiable of order $\alpha$ with
\begin{equation*}
f^{\diamondsuit ^{\alpha }}\left( t\right) =\gamma _{1}\left( t\right)
f^{\Delta ^{\alpha }}\left( t\right) +\gamma _{2}\left( t\right) f^{\nabla
^{\alpha }}\left( t\right)
\end{equation*}
for each $t\in \mathbb{T}_{\kappa }^{\kappa }$, where
\begin{equation*}
\gamma _{1}\left( t\right) :=\lim_{s\rightarrow t}\left[ \frac{\sigma \left(
t\right) -s}{\sigma \left( t\right) +2t-2s-\rho \left( t\right) }\right]
^{\alpha }
\end{equation*}
and
\begin{equation*}
\gamma _{2}\left( t\right) :=\lim_{s\rightarrow t}\left[ \frac{\left(
2t-s\right) -\rho \left( t\right) }{\sigma \left( t\right) +2t-2s-\rho
\left( t\right) }\right] ^{\alpha }.
\end{equation*}
\end{proposition}

\begin{proof}
Note that
\begin{equation*}
\begin{split}
f^{\diamondsuit ^{\alpha }}\left( t\right)
&=\lim_{s\rightarrow t}\frac{f^{\sigma }\left( t\right)
-f\left( s\right) +f\left( 2t-s\right)
-f^{\rho}\left( t\right) }{\left[ \sigma \left( t\right)
+2t-2s-\rho \left( t\right)\right] ^{\alpha }} \\
&=\lim_{s\rightarrow t}\Biggl(\frac{\left[ \sigma \left( t\right)
-s\right]^{\alpha }}{\left[ \sigma \left( t\right) +2t-2s
-\rho \left( t\right) \right]^{\alpha }}\frac{f^{\sigma }\left( t\right)
-f\left( s\right) }{\left[\sigma \left( t\right) -s\right] ^{\alpha }} \\
&\qquad +\frac{\left[ \left( 2t-s\right) -\rho \left( t\right) \right] ^{\alpha }}{
\left[ \sigma \left( t\right) +2t-2s-\rho \left( t\right) \right] ^{\alpha }}
\frac{f\left( 2t-s\right) -f^{\rho }\left( t\right) }{\left[ \left(
2t-s\right) -\rho \left( t\right) \right] ^{\alpha }}\Biggr) \\
&=\lim_{s\rightarrow t}\Biggl( \left[ \frac{\sigma \left( t\right) -s}{
\sigma \left( t\right) +2t-2s-\rho \left( t\right) }\right] ^{\alpha
}f^{\Delta }\left( t\right) \\
&\qquad +\left[ \frac{\left( 2t-s\right) -\rho \left(
t\right) }{\sigma \left( t\right) +2t-2s-\rho \left( t\right) }\right]
^{\alpha }f^{\nabla }\left( t\right) \Biggr).
\end{split}
\end{equation*}
If $t\in \mathbb{T}$ is dense, then
\begin{equation*}
\gamma_{1}\left( t\right) =\lim_{s\rightarrow t}\left[ \frac{\sigma
\left( t\right) -s}{\sigma \left( t\right) +2t-2s-\rho \left( t\right)}
\right]^{\alpha }
=\lim_{s\rightarrow t}\left[ \frac{t-s}{2t-2s}\right] ^{\alpha }\\
=\frac{1}{2^{\alpha }}
\end{equation*}
and
\begin{equation*}
\gamma _{2}\left( t\right) =\lim_{s\rightarrow t}\left[ \frac{\left(
2t-s\right) -\rho \left( t\right) }{\sigma \left( t\right) +2t-2s-\rho
\left( t\right) }\right]^{\alpha }
=\lim_{s\rightarrow t}\left[ \frac{t-s}{2t-2s}\right] ^{\alpha } \\
=\frac{1}{2^{\alpha }}.
\end{equation*}
On the other hand, if $t\in \mathbb{T}$ is not dense, then
\begin{equation*}
\gamma _{1}\left( t\right)
=\lim_{s\rightarrow t}\left[ \frac{\sigma
\left( t\right) -s}{\sigma \left( t\right) +2t-2s-\rho \left( t\right)}
\right]^{\alpha}
=\left[ \frac{\sigma \left( t\right) -t}{\sigma \left( t\right) -\rho
\left( t\right) }\right] ^{\alpha }
\end{equation*}
and
\begin{equation*}
\gamma _{2}\left( t\right) =\lim_{s\rightarrow t}\left[ \frac{\left(
2t-s\right) -\rho \left( t\right) }{\sigma \left( t\right) +2t-2s-\rho
\left( t\right) }\right] ^{\alpha}
=\lim_{s\rightarrow t}\left[ \frac{t-\rho \left( t\right) }{\sigma \left(
t\right) -\rho \left( t\right) }\right] ^{\alpha }.
\end{equation*}
Hence, functions $\gamma _{1},\, \gamma _{2}:\mathbb{T}\rightarrow
\mathbb{R}$ are well defined and, if $f$ is delta and nabla differentiable, then
$f^{\diamondsuit }\left( t\right) =\gamma _{1}\left( t\right) f^{\Delta
}\left( t\right) +\gamma _{2}\left( t\right) f^{\nabla }\left( t\right)$.
\end{proof}

\begin{remark}
Suppose that $f$ is delta and nabla fractional differentiable of order $\alpha$.
If point $t\in \mathbb{T}_{\kappa }^{\kappa }$ is right-scattered and left-dense,
then its fractional symmetric derivative of order $\alpha$ is equal
to its delta fractional derivative of order $\alpha$.
If $t$ is left-scattered and right-dense, then its symmetric fractional derivative
of order $\alpha$ is equal to its nabla fractional derivative of order $\alpha$.
\end{remark}

Due to Proposition~\ref{delta:nabla}, we can now define
a symmetric integral of noninteger order.

\begin{definition}[The symmetric fractional integral]
\label{def:intFracCauchy sym}
Assume function $f:\mathbb{T}\rightarrow \mathbb{R}$ is simultaneously
rd- and ld-continuous. Let $a,b\in \mathbb{T}$ and
$F^{\Delta^{\beta }}(t) =\int f(t)\Delta^{\beta}t$ and
$F^{\nabla^{\beta}}(t) =\int f(t)\nabla^{\beta }t$
denote the indefinite delta and nabla fractional integrals
of $f$ of order $\beta$, respectively. Then we define the Cauchy
symmetric fractional integral of $f$ of order $\beta \in ]0,1]$ by
\begin{equation*}
\begin{split}
\int_{a}^{b}f(t)\diamondsuit ^{\beta }t
&=\left. \gamma _{1}\left(t\right)
F^{\Delta ^{\beta }}(t)\right\vert _{a}^{b}+\left.
\gamma_{2}\left( t\right) F^{\nabla ^{\beta }}(t)\right\vert _{a}^{b}\\
&=\gamma _{1}\left( b\right) F^{\Delta ^{\beta }}(b)-\gamma _{1}\left(
a\right) F^{\Delta ^{\beta }}(a)+\gamma _{2}\left( b\right) F^{\nabla
^{\beta }}(b)-\gamma _{2}\left( a\right) F^{\nabla ^{\beta }}(a).
\end{split}
\end{equation*}
\end{definition}

Finally, we present some algebraic properties
of the symmetric fractional integral.

\begin{theorem}
\label{int:sym:props}
Let $a,b,c\in \mathbb{T}$ and $\lambda \in \mathbb{R}$.
If $f,g\in \mathcal{C}_{ld}$ and $f,g\in \mathcal{C}_{rd}$
with $0 \leq \beta \leq 1$, then
\begin{enumerate}
\item[(i)] $\displaystyle\int_{a}^{b}[f(t)+g(t)]\diamondsuit ^{\beta}t
=\int_{a}^{b}f(t)\diamondsuit ^{\beta }t+\int_{a}^{b}g(t)\diamondsuit^{\beta }t$;

\item[(ii)] $\displaystyle\int_{a}^{b}(\lambda f)(t)\diamondsuit^{\beta}t
=\lambda \int_{a}^{b}f(t)\diamondsuit ^{\beta }t$;

\item[(iii)] $\displaystyle\int_{a}^{b}f(t)\diamondsuit^{\beta}t
=-\int_{b}^{a}f(t)\diamondsuit ^{\beta }t$;

\item[(iv)] $\displaystyle\int_{a}^{b}f(t)\diamondsuit^{\beta}t
=\int_{a}^{c}f(t)\diamondsuit ^{\beta }t
+\int_{c}^{b}f(t)\diamondsuit^{\beta }t$;

\item[(v)] $\displaystyle\int_{a}^{a}f(t)\diamondsuit ^{\beta }t=0$.
\end{enumerate}
\end{theorem}

\begin{proof}
Equalities (i)--(v) follow from Definition~\ref{def:intFracCauchy sym}
and analogous properties of the nabla and delta fractional integrals
(cf. Theorem~\ref{T4}).
\end{proof}

% -------------------------------------------

\section{Conclusion}
\label{sec:Conc}

Fractional calculus, that is, the study of differentiation and integration
of noninteger order, is here extended, via the recent and powerful calculus
on time scales, to include, in a single theory, discrete, continuous
and hybrid fractional calculi. Both nonsymmetric and symmetric fractional
derivatives and integrals on an arbitrary nonempty closed subset of the real
numbers are introduced, and their fundamental properties derived.
It is shown that a function may be fractional differentiable
but not differentiable; and that a function may be symmetric
fractional differentiable but not fractional differentiable.
A relation between the nonsymmetric and symmetric fractional
derivatives is also derived. In particular,
our time-scale symmetric fractional calculus
of order $\alpha\in]0,1]$ gives, when $\alpha =1$,
the recent results of \cite{Cruz:2}; while for $\alpha =1$
the delta and nabla nonsymmetric fractional calculi
reduce to the the usual delta and nabla calculus on time scales, respectively.

We have only introduced some fundamental concepts and proved some
basic properties. Much remains to be done in order to develop the theory
here initiated. In particular, it would be interesting to investigate
the usefulness of the new fractional calculi in applications
to real world problems, where the time scale is partially
continuous and partially discrete with a time-varying graininess function.
This and other questions will be subject of future research.

% -------------------------------------------

\section*{Acknowledgments}

This work is part of first author's PhD,
which is carried out at Sidi Bel Abbes University, Algeria.
It was partially supported by the \emph{Center for Research
and Development in Mathematics and Applications} (CIDMA)
and the \emph{Portuguese Foundation for Science and Technology} (FCT),
within project UID/MAT/04106/2013. The authors are grateful to two
anonymous referees for important comments and suggestions.

% -------------------------------------------

% -------------------------------------------------------------


\begin{thebibliography}{99}

\bibitem{book:Benchohra}
S. Abbas, M. Benchohra\ and\ G. M. N'Gu\'er\'ekata,
{\it Topics in fractional differential equations},
Developments in Mathematics, 27,
Springer, New York, 2012.

\bibitem{MR0247389}
R. P. Agarwal,
Certain fractional $q$-integrals and $q$-derivatives,
Proc. Cambridge Philos. Soc. {\bf 66} (1969), 365--370.

\bibitem{MR2963764}
M. H. Annaby\ and\ Z. S. Mansour,
{\it $q$-fractional calculus and equations},
Lecture Notes in Mathematics, 2056, Springer, Heidelberg, 2012.

\bibitem{MR2350094}
F. M. Atici\ and\ P. W. Eloe,
Fractional $q$-calculus on a time scale,
J. Nonlinear Math. Phys. {\bf 14} (2007), no.~3, 333--344.

\bibitem{Hilger}
B. Aulbach\ and\ S. Hilger,
A unified approach to continuous and discrete dynamics,
in {\it Qualitative theory of differential equations (Szeged, 1988)}, 37--56,
Colloq. Math. Soc. J\'anos Bolyai, 53, North-Holland, Amsterdam, 1990.

\bibitem{PhD:Bastos}
N. R. O. Bastos,
{\it Fractional calculus on time scales},
PhD thesis (supervisor: D. F. M. Torres),
Doctoral Programme in Mathematics and Applications (PDMA Aveiro--Minho),
University of Aveiro, 2012.
{\tt arXiv:1202.2960}

\bibitem{MR2728463}
N. R. O. Bastos, R. A. C. Ferreira\ and\ D. F. M. Torres,
Necessary optimality conditions for fractional difference
problems of the calculus of variations,
Discrete Contin. Dyn. Syst. {\bf 29} (2011), no.~2, 417--437.
{\tt arXiv:1007.0594}

\bibitem{MyID:179}
N. R. O. Bastos, R. A. C. Ferreira\ and\ D. F. M. Torres,
Discrete-time fractional variational problems,
Signal Process. {\bf 91} (2011), no.~3, 513--524.
{\tt arXiv:1005.0252}

\bibitem{MR2800417}
N. R. O. Bastos, D. Mozyrska\ and\ D. F. M. Torres,
Fractional derivatives and integrals on time scales
via the inverse generalized Laplace transform,
Int. J. Math. Comput. {\bf 11} (2011), J11, 1--9.
{\tt arXiv:1012.1555}

\bibitem{MyID:296}
N. Benkhettou, A. M. C. Brito da Cruz\ and\ D. F. M. Torres,
A fractional calculus on arbitrary time scales:
fractional differentiation and fractional integration,
Signal Process. {\bf 107} (2015), 230--237.
{\tt arXiv:1405.2813}

\bibitem{Bohner:1}
M. Bohner\ and\ A. Peterson,
{\it Dynamic equations on time scales},
Birkh\"auser Boston, Boston, MA, 2001.

\bibitem{Bohner:2}
M. Bohner\ and\ A. Peterson,
{\it Advances in dynamic equations on time scales},
Birkh\"auser Boston, Boston, MA, 2003.

\bibitem{Cruz:1}
A. M. C. Brito da Cruz,
{\it Symmetric quantum calculus}, 
PhD thesis (supervisor: D. F. M. Torres; co-supervisor: N. Martins),
Doctoral Programme in Mathematics and Applications (PDMA Aveiro--Minho),
University of Aveiro, 2012.
{\tt arXiv:1306.1327}

\bibitem{Cruz:2}
A. M. C. Brito da Cruz, N. Martins\ and\ D. F. M. Torres,
Symmetric differentiation on time scales,
Appl. Math. Lett. {\bf 26} (2013), no.~2, 264--269.
{\tt arXiv:1209.2094}

\bibitem{MR3110279}
A. M. C. Brito da Cruz, N. Martins\ and\ D. F. M. Torres,
A symmetric quantum calculus,
in {\it Differential and difference equations with applications},
359--366, Springer Proc. Math. Stat., 47, Springer, New York, 2013.
{\tt arXiv:1112.6133}

\bibitem{MR3110294}
A. M. C. Brito da Cruz, N. Martins\ and\ D. F. M. Torres,
A symmetric N\"orlund sum with application to inequalities,
in {\it Differential and difference equations with applications},
495--503, Springer Proc. Math. Stat., 47, Springer, New York, 2013.
{\tt arXiv:1203.2212}

\bibitem{MR3031158}
A. M. C. Brito da Cruz, N. Martins\ and\ D. F. M. Torres,
Hahn's symmetric quantum variational calculus,
Numer. Algebra Control Optim. {\bf 3} (2013), no.~1, 77--94.
{\tt arXiv:1209.1530}

\bibitem{Cruz:3}
A. M. C. Brito da Cruz, N. Martins\ and\ D. F. M. Torres,
The diamond integral on time scales,
Bull. Malays. Math. Sci. Soc.,
DOI: 10.1007/s40840-014-0096-7.
{\tt arXiv:1306.0988}

\bibitem{Ck}
S. K. Choi\ and\ N. J. Koo,
Dynamic equations on time scales,
Trends in Mathematics {\bf 7} (2004), no.~2, 63--69.

\bibitem{Dryl}
M. Dryl\ and\ D. F. M. Torres,
The delta-nabla calculus of variations for composition functionals on time scales,
Int. J. Difference Equ. {\bf 8} (2013), no.~1, 27--47.
{\tt arXiv:1211.4368}

\bibitem{MR2809039}
R. A. C. Ferreira\ and\ D. F. M. Torres,
Fractional $h$-difference equations arising from the calculus of variations,
Appl. Anal. Discrete Math. {\bf 5} (2011), no.~1, 110--121.
{\tt arXiv:1101.5904}

\bibitem{MR3243775}
E. Girejko\ and\ D. Mozyrska,
Semi-linear fractional systems with Caputo type multi-step differences,
Carpathian J. Math. {\bf 30} (2014), no.~2, 187--195.

\bibitem{Kac}
V. Kac\ and\ P. Cheung,
{\it Quantum calculus},
Universitext, Springer, New York, 2002.

\bibitem{MR2798773}
T. Kaczorek,
{\it Selected problems of fractional systems theory},
Lecture Notes in Control and Information Sciences, 411, Springer, Berlin, 2011.

\bibitem{Kilbas}
A. A. Kilbas, H. M. Srivastava\ and\ J. J. Trujillo,
{\it Theory and applications of fractional differential equations},
North-Holland Mathematics Studies, 204, Elsevier, Amsterdam, 2006.

\bibitem{MR2562284}
A. B. Malinowska\ and\ D. F. M. Torres,
On the diamond-alpha Riemann integral and mean value theorems on time scales,
Dynam. Systems Appl. {\bf 18} (2009), no.~3-4, 469--481.
{\tt arXiv:0804.4420}

\bibitem{MR2994055}
N. Martins\ and\ D. F. M. Torres,
Necessary optimality conditions for higher-order infinite
horizon variational problems on time scales,
J. Optim. Theory Appl. {\bf 155} (2012), no.~2, 453--476.
{\tt arXiv:1204.3329}

\bibitem{Metzler}
R. Metzler\ and\ J. Klafter,
The restaurant at the end of the random walk: recent developments
in the description of anomalous transport by fractional dynamics,
J. Phys. A {\bf 37} (2004), no.~31, R161--R208.

\bibitem{MR3178922}
D. Mozyrska,
Multiparameter fractional difference linear control systems,
Discrete Dyn. Nat. Soc. {\bf 2014}, Art. ID 183782, 8~pp.

\bibitem{MR3270279}
D. Mozyrska\ and\ E. Paw\l uszewicz,
Controllability of $h$-difference linear control systems
with two fractional orders,
Internat. J. Systems Sci. {\bf 46} (2015), no.~4, 662--669.

\bibitem{MR2601883}
D. Mozyrska, E. Paw\l uszewicz\ and\ D. F. M. Torres,
The Riemann-Stieltjes integral on time scales,
Aust. J. Math. Anal. Appl. {\bf 7} (2010), no.~1, Art. 10, 14~pp.
{\tt arXiv:0903.1224}

\bibitem{Samko}
S. G. Samko, A. A. Kilbas\ and\ O. I. Marichev,
{\it Fractional integrals and derivatives},
translated from the 1987 Russian original,
Gordon and Breach, Yverdon, 1993.

\bibitem{Sheng}
Q. Sheng, M. Fadag, J. Henderson\ and\ J. M. Davis,
An exploration of combined dynamic derivatives on time scales and their applications,
Nonlinear Anal. Real World Appl. {\bf 7} (2006), no.~3, 395--413.

\bibitem{Thomson}
B. S. Thomson,
{\it Symmetric properties of real functions},
Monographs and Textbooks in Pure and Applied Mathematics,
183, Dekker, New York, 1994.

\end{thebibliography}
\end{document}